\documentclass{amsart}
\usepackage{amsfonts}
\usepackage{amssymb}
\usepackage[mathscr]{eucal} %% caligrafic letters with \mathscr{}
\usepackage{times}
\usepackage{verbatim} %% \begin{comment} ... \end{comment}

\newtheorem{prop}{\bf Proposition}[section]
\newtheorem{defi}[prop]{\bf Definition}
\newtheorem{theo}[prop]{\bf Theorem}
\def\A{\mathscr{A}}
\def\B{\mathscr{B}}
\def\F{\mathbb{F}}
\def\L{\mathcal{L}}
\def\Z{\mathbb{Z}}
\def\X{\mathfrak{X}}
\newcommand{\ii}{\textbf{i}}
\def\C{\mathscr{C}}
\def\aut{\mathop{{\rm Aut}}}
\def\supp{\mathop{\hbox{\rm  Supp}}}
\def\deg{\mathop{\hbox{\rm  deg}}}
\def\Diag{\mathop{\hbox{\rm  Diag}}}
\def\diag{\mathop{\hbox{\rm  diag}}}
\def\stab{\mathop{\hbox{\rm  Stab}}}
\def\norm{\mathop{\hbox{\rm  Norm}}}
\def\tr{\mathop{\hbox{\rm  tr}}}
\def\W{\mathscr{W}}
\def\alb{\mathbb{A}}
\def\T{\mathcal{T}}
\def\J{\mathcal{J}}
\def\der{\mathop{\hbox{\rm Der}}}
\def\slin{\mathfrak{sl}}
\def\glin{\mathfrak{gl}}
\def\sp{\mathfrak{sp}}
\def\g{\mathfrak{g}}
\def\f{\mathfrak{f}}
\def\e{\mathfrak{e}}
\def\Ad{\mathop{\rm Ad}}
\def\ad{\mathop{\rm ad}}
\def\End{\mathop{\rm End}}
\def\Cent{\mathop{\rm Cent}}
\def\Mat{\mathop{\rm Mat}}

\def\id{\mathop{\rm id}}

\def\SL{\mathop{\rm SL}}
\def\PSL{\mathop{\rm PSL}}
\def\GL{\mathop{\rm GL}}

\def\Sp{\mathop{\rm SP}}

\def\span#1{\langle #1 \rangle}

\def\fix{\mathop{\rm fix}}

\title{Weyl groups of the fine gradings    on $\e_6$}

\author[D. Aranda]{Diego Aranda Orna}
\email{daranda@unizar.es}
\address{Dpto. de Matem\'aticas\\ Facultad de Ciencias\\
Universidad de Zaragoza, 50009, Zaragoza, Spain}
\thanks{\ The first author is partially supported by the grant
MTM 2010-18370-C04-02 and by the Diputaci\'on
General de Arag\'on and Fondo Social Europeo.
The second author is partially supported by the
MCYT grant MTM 2010-15223 and by the Junta de Andaluc\'{\i}a grants FQM-336, FQM-1215
and FQM-2467.
The third author is partially supported by   the Ministero dell'Istruzione, dell'Universit\`{a} e della
Ricerca (MIUR) with a grant for the Ph.D. at the Universit\`{a} del Salento}

\author[C. Draper]{Cristina Draper Fontanals}
\email{cdf@uma.es}
\address{Dpto. de Matem\'atica Aplicada\\Escuela de las Ingenier\'{\i}as, Universidad de M\'alaga\\
Ampliaci\'{o}n Campus de Teatinos, 29071, M\'alaga, Spain}

\author[V. Guido]{Valerio Guido}
\email{valerio.guido@unisalento.it}
\address{Dpto. di Matematica E. De Giorgi\\ Universit\`{a} del Salento\\
Via Provinciale Lecce-Arnesano, 73100, Lecce, Italy}

\begin{document}

\maketitle

\begin{abstract}
The Weyl groups of the fine gradings with infinite universal grading group on $\mathfrak{e}_6$ are given.
\end{abstract}

%\maketitle


\noindent {\bf MSC2010 Subject Classification:} 17B25, %17B25  	Exceptional (super)algebras
17B70. %17B70  	Graded Lie (super)algebras.
%17B40  	Automorphisms, derivations, other operators

%%%%%%%%%%%%%%%%%%%%%%%%%%%%%%%%%%%%%%%%%%%%%%%%%%%%
%%%
%%%%%%%%%%%%%%%%%%%%%%%%%    INTRODUCTION  %%%%%%%%%%%
%%%%%%%%%%%%%%%%%%%%%%%%%%%%%%%%%%%%%%%%%%%%%%%%%%%%%%

\section{Introduction}

%%%%%%%%%%%%%%%%%%%%%%%%%%%%%%%%%%%%%%%%%%%%%%%%%%%%%%%%%%%%%%%%%%%%%%%%%%%%%%%
%REVISI\'{O}N HIST\'{O}RICA DEL CONCEPTO

If $\mathcal{G}$ denotes a connected algebraic group, defined over an algebraically closed field $\F$, and $T$ is a maximal torus of  $\mathcal{G}$, because of the rigidity of tori, the quotient $N_{\mathcal{G}}(T)/C_{\mathcal{G}}(T)$ is a finite group, called the Weyl group of $\mathcal{G} $ (since all the maximal tori are conjugate, all their Weyl groups are isomorphic). Here $N_{\mathcal{G}}(T)$ is the normalizer and $C_{\mathcal{G}}(T)$ is the centralizer of $T$ in $\mathcal{G}$.   
In case  $\mathcal{G}$  is a reductive group, this Weyl group is isomorphic to the  Weyl group of the root system.  Recall that a Weyl group of a root system $\Phi$ is the subgroup of the isometry group of $\Phi$ generated by reflections through the hyperplanes orthogonal to the roots, and as such is a finite reflection group. Abstractly, Weyl groups are finite Coxeter groups, and are important examples of these.
The Weyl group of a semisimple finite-dimensional Lie algebra is the Weyl group of its root system.
There is no need to emphasize the importance of Weyl groups in the
theory of Lie algebras, and their
important role in representation theory. Thus, a lot of generalizations of this concept have arisen, among them the Weyl group of a Lie grading.

%%%%%%%%%%%%%%%%%%%%%%%%%%%%%%%%%%%%%%%%%%%%%%%%%%%%%%%%%%%%%%%%%%%%%%%%%%%%%%%
%PRECEDENTES

This notion was introduced by Patera and Zassenhaus in  \cite{LGI}, trying to generalize the groups related with the Cartan decomposition of a  complex simple Lie algebra. That work can be considered as the beginning of the theory of the gradings on Lie algebras. Until that moment, the gradings usually were considered over cyclic groups, with the obvious exception of the Cartan decomposition (root decomposition). They proposed the search of the remaining fine gradings as a way of looking for an alternative approach to the structure theory and
representation theory. Moreover, these fine gradings  provide immediately basis with interesting properties, as proved in \cite[Proposition 10]{f4}. In a particular case,   \cite[Lemma~1]{e6} shows that every element in any basis of homogeneous elements of a fine grading by a finite group is %automatically
directly semisimple.

The task of computing the Weyl group of one fine grading on a concrete Lie algebra was done in \cite{checosWeyl}, for the Lie algebra $\slin(p^2,\mathbb{C})$, $p$ being a prime, and for the fine grading given as a tensor product of the Pauli matrices
of the same order $p$. In authors' words, the study of the normalizer of the MAD-group corresponding to a fine grading offers the most important tool for describing symmetries in the system of nonlinear equations
connected with a contraction of the Lie algebra. This technique is applied in \cite{normPauli} for obtaining all the graded contractions
of the Lie algebra $\slin_3(\mathbb{C})$ arising from the Pauli grading.
Afterwards, the same job has been boarded in \cite{EK11} and \cite{EK12} for a variety of algebras: matrix algebras, octonion algebras, Albert algebras, and Lie algebras of types $A$, $B$, $C$ and $D$.
As  $\frak{g}_2$
is the derivation algebra of the octonion algebra, and $\frak{f}_4$
is the derivation algebra of the Albert algebra, our goal will be to determine  these Weyl groups for $\e_6$, the following simple Lie algebra in size, and the only one of the series-E for which the fine gradings are completely classified up to equivalence (\cite{e6}).

Thus, we wish to investigate the symmetries of  the splittings on $\e_6$.
This algebra   has the striking quantity of $14$ fine gradings   up to equivalence, whose universal groups and types will be brought  in Section~\ref{sec_graduaciones4}. It is not our purpose to study the $7$ fine gradings whose universal group is  finite, due to the fact that the work  \cite{elchino}  has  almost completed the problem of finding the Weyl groups in such cases. Also, the normalizer of the MAD-group of type $\Z_3^4$ is described in \cite{Viruannals} in detail, while determining properties of the elementary $3$-subgroups of $\aut(\e_6)$. Consequently, we will restrict our attention to the infinite cases different from the Cartan grading.

In the case   of the Cartan grading, its Weyl group is of course very well known, as well as its applications to other branches of Mathematics, as finite geometry and incidence structures.
Namely, the Weyl group of $\e_6$ is closely related with the automorphism group of the $27$ lines on the smooth cubic surface. This connection was recognized through the relation with del Pezzo surfaces of degree three (\cite{Coxeter}).
Also \cite{Manivel} discusses about relations between classical configurations of lines and the (complex) Lie algebra $\e_6$.\smallskip

%%%%%%%%%%%%%%%%%%%%%%%%%%%%%%%%%%%%%%%%%%%%%%%%%%%%%%%%%%%%%%%%%%%%%%%%%%%%%%%
%ESTRUCTURA

The paper is structured as follows.
Section~2 gives the basic definitions  about group gradings and about several groups related to them, including the notion of  Weyl group of a grading  as a generalization of the classical concept.
In Section~3 we have compiled some algebraic constructions of the Lie algebra  $\e_6$, which will be used in further sections. This list of models begins with the most famous, the unified Tits' construction, but it encloses other ones not explicitly appeared in the literature. Also this section reviews some of the standard facts of the  algebraic structures involved in the mentioned constructions: composition algebras, symmetric composition algebras and the exceptional Jordan algebra, the Albert algebra. Section~4 provides a detailed description of the fine gradings on $\e_6$ with infinite universal groups.
Finally, in the fifth section
 our main results are stated and proved: they consist of the computation of
 the Weyl groups of each of the above gradings.
They are  summarized  
 in Theorem~\ref{th_maintheorem}.\smallskip

%%%%%%%%%%%%%%%%%%%%%%%%%%%%%%%%%%%%%%%%%%%%%%%%%%%%%%%%%%%%%%%%%%%%%%%%%%%%%%%%
%CONJETURA
In light  of the obtained results, we can think of the following question:

\textbf{Conjecture:} If $Q$ is a maximal abelian diagonalizable group of   the algebraic group
$\mathcal{G}=\aut(L)$, for $L$ a semisimple Lie algebra over an algebraically closed field $\F$ of characteristic zero,
then two elements $q_1,q_2\in Q$ are conjugated in $ \mathcal{G}$ if and only if they are conjugated in the normalizer of $Q$.

This result is well known for the case in which $Q$ is a maximal torus of $ \mathcal{G}$ (see for instance \cite[p.~ 99]{Carter}), that is, in case the diagonalization produced by $Q$ is the Cartan grading on $L$. The proof is not trivial at all and it is related with some deep topics, as the Bruhat decomposition or the BN-pairs.
As far as we know, the question for an arbitrary MAD-group $Q$ is at present far from being solved.
A direct proof of this conjecture would shorten considerably our work.
An affirmative answer for the case $L=\e_6$ is a consequence of the results of this paper, and the cases $\frak{f}_4$ and $\frak{g}_2$ can be concluded from \cite{EK11}. This is unlikely to be a coincidence.

%%%%%%%%%%%%%%%%%%%%%%%%%%%%%%%%%%%%%%%%%%%%%%%%%%%%%%
%%%%%%%%%%%%%%%%%%%%%%%%%    SECTION 2  %%%%%%%%%%%%%%
%%%%%%%%%%%%%%%%%%%%%%%%%%%%%%%%%%%%%%%%%%%%%%%%%%%%%%

\section{Basics on gradings and their Weyl groups}\label{sec_defsbasicas2}

Let $\A$ be a finite-dimensional algebra (not necessarily associative) over a field $\F$, and $G$ an abelian group.

\defi\rm A \emph{$G$-grading} $\Gamma$ on $\A$ is a vector space decomposition
$$
 \Gamma: \A = \bigoplus_{g\in G} \A_g
 $$
such that
$$
 \A_g \A_h\subseteq \A_{g+h} \quad {\textnormal{ for all }} g,h\in G.
 $$

Fixed such a decomposition, the algebra $\A$ will be called a \emph{$G$-graded algebra}, the subspace $\A_g$ will be referred to as \emph{homogeneous component of degree $g$} and its nonzero elements will be called the \emph{homogeneous elements of degree $g$}. The \emph{support} is the set $\supp \Gamma :=\{g\in G\mid \A_g\neq 0\}$.

\defi\rm If $\Gamma\colon \A=\oplus_{g\in G} \A_g$ and $\Gamma'\colon \A=\oplus_{h\in H} \A_{h}$ are gradings over two abelian groups $G$ and $H$, $\Gamma$ is said to be a \emph{refinement} of $\Gamma'$ (or $\Gamma'$ a \emph{coarsening} of $\Gamma$) if for any  $g\in G$ there is $h\in H$ such that $\A_g\subseteq \A_{h}$. In other words, any homogeneous component of $\Gamma'$ is the direct sum of some homogeneous components of $\Gamma$. A refinement is \emph{proper} if some inclusion $\A_g\subseteq \A_{h}$ is proper. A grading is said to be \emph{fine} if it admits no proper refinement.

\defi\rm Let $\Gamma$ be a $G$-grading on $\A$ and $\Gamma'$ an $H$-grading on $\B$, with supports, respectively, $S$ and $T$. Then $\Gamma$ and $\Gamma'$ are said to be \emph{equivalent} if there is an algebra isomorphism  $\varphi\colon\A \rightarrow \B$ and a bijection $\alpha\colon  S \rightarrow T$ such that $\varphi(\A_s)=\B_{\alpha(s)}$ for all $s\in S$. Any such $\varphi$ is called an \emph{equivalence} of $\Gamma$ and $\Gamma'$.

\smallskip

The study of gradings is based on classifying fine gradings up to equivalence, because any grading is obtained as a coarsening of some fine one.
A  useful invariant of a grading is given by the dimensions of the components: the \emph{type} of a grading $\Gamma$ is the sequence of numbers $(h_1,\ldots, h_r)$ where $h_i$ is the number of homogeneous components of dimension $i$, with $i=1,\ldots, r$ and $h_r\neq 0$. Obviously, $\dim \A = \sum_{i=1}^r ih_i$.

It can always be assumed that $\supp \Gamma$ generates $G$ (replacing $G$ with a smaller group, if necessary), but there are, in general, many other groups $H$ such that the vector space decomposition $\Gamma$ can be realized as an $H$-grading. According to \cite{LGI}, there is one distinguished group among them:

\defi\rm Suppose that $\Gamma$ admits a realization as a $G_0$-grading for some abelian group $G_0$. We will say that $G_0$ is a \emph{universal grading group of $\Gamma$} if for any other realization of $\Gamma$ as a $G$-grading, there exists a unique group homomorphism $G_0 \rightarrow G$ that restricts to the identity on $\supp \Gamma$. We denote this group by $U(\Gamma)$, which   always exists
and besides it is     
unique up to isomorphism   depending only on the equivalence class of $\Gamma$ (see \cite{Koc09}). We will only deal with gradings over their universal grading groups.

\medskip

From now on, consider a finite-dimensional Lie algebra $\L$. The ground field $\F$ will be supposed to be algebraically closed and of characteristic zero throughout this work. With these hypothesis, the automorphism group of $\L$ is an algebraic linear group. Thus, following \cite[\S 3, p.104]{enci}, it is useful to look at gradings focusing on quasitori of the automorphism group $\aut( \L)$, where quasitori are diagonalizable groups, or equivalently, direct products of tori with abelian finite groups.

If $\L=\oplus_{g\in G} \L_g$ is a $G$-grading, the map $\psi\colon  \X(G):=\text{hom} (G,\F^\times) \rightarrow
\aut(\L)$  applying each $\alpha \in \X(G)$ to the automorphism $\psi_{\alpha}\colon \L \rightarrow \L$ given by $\L_g\ni x \mapsto \psi_{\alpha}(x) :=\alpha (g) x$ is a group homomorphism. Since $G$ is finitely generated, $\psi(\X(G))$ is an algebraic quasitorus. And conversely, if $Q$ is a quasitorus and $\psi\colon  Q \rightarrow \aut (\L)$ is a homomorphism, $\psi(Q)$ consists of semisimple automorphisms and we have a $\X(Q)$-grading $\L=\oplus_{g\in\X(Q)} \L_g$ given by $\L_g=\{x\in \L\mid \psi(q)(x)=g(q)x \, \forall q \in Q\}$, with $\X(Q)$ a finitely generated abelian group.

In case $Q$ is a torus $T$ of the automorphism group of $\L$, this means that there exists a group isomorphism $(\F^\times)^s\to T$,
$(\alpha_1,\dots,\alpha_s)\mapsto t_{\alpha_1,\dots,\alpha_s}$, thus $\Z^s\cong\X(T)$ by means of $(n_1,\dots,n_s)\mapsto \psi_{n_1,\dots,n_s}\colon T\to \F^\times$, with
$\psi_{n_1,\dots,n_s}(t_{\alpha_1,\dots,\alpha_s})=\alpha_1^{n_1}\dots\alpha_s^{n_s}$, and the $\Z^s$-grading produced on $\L$ is given by
\begin{equation}\label{eq_torosconisomorfismoprefijado}
\L_{(n_1,\dots,n_s)}=\{x\in\L\mid t_{\alpha_1,\dots,\alpha_s}(x)=\alpha_1^{n_1}\dots\alpha_s^{n_s}x\quad\forall \alpha_i\in\F^\times,\forall i=1,\dots,s\}.
\end{equation}
A grading is \emph{toral} if there exists a torus $T$ of the automorphism group of $\L$ containing the quasitorus that {produces} the grading. This is equivalent to the grading is a coarsening of the root decomposition of a semisimple Lie algebra $\L$ relative to some Cartan subalgebra.  \smallskip

Following \cite{LGI}, we can consider three distinguished subgroups of $\aut (\L)$ associated to a $G$-grading $\Gamma\colon  \L= \oplus_{g\in G} \L_g$.

\defi\rm The subgroup of $\aut(\L)$ of all automorphisms of $\L$ that permute the components of $\Gamma$ is called the \emph{automorphism group of $\Gamma$} and  denoted by $\aut( \Gamma)$. Each $\varphi \in \aut (\Gamma)$ determines a self-bijection $ \alpha_{\varphi}\colon S\to S$ of the support $S=\supp(\Gamma)$ such that $\varphi(\L_s)=\L_{\alpha_{\varphi}(s)}$ for all $s\in S$.

\defi\rm The \emph{stabilizer of $\Gamma$}, denoted by $\stab (\Gamma)$, is the set of automorphisms of $\L$ that leave invariant the homogeneous components of the grading, that is, the group of automorphisms of $\mathcal{L}$ as graded algebra. It coincides with the kernel of the homomorphism $\aut (\Gamma) \rightarrow \text{Sym} (S)$ given by $\varphi \mapsto \alpha_{\varphi}$.

\defi\rm  The \emph{diagonal group} of the grading,  denoted by $\Diag (\Gamma)$, is the set of automorphisms of $\L$ such that each $\L_g$ is contained in some eigenspace. This group is a quasitorus isomorphic to $\X(U(\Gamma))$, and the related  $\X(\Diag (\Gamma))$-grading (the eigenspace  decomposition of $\mathcal{L}$ relative to
 $\Diag (\Gamma)$) is equivalent to $\Gamma$, under our assumptions on the field.
 Moreover, the group $\stab (\Gamma)$ is the centralizer of $\Diag (\Gamma)$, and $\aut (\Gamma)$ is its normalizer.
\smallskip

The grading $\Gamma$ is fine if and only if $\Diag(\Gamma)$ is a Maximal Abelian Diagonalizable group,  usually called a \emph{MAD-group}. Moreover the number of conjugacy classes of MAD-groups in $\aut(\L)$ agrees with the number of equivalence classes of fine gradings on $\L$.

\defi\rm
The \emph{Weyl group} of $\Gamma$, denoted by $\W(\Gamma)$, is defined as the quotient group $\aut(\Gamma)/ \stab(\Gamma)$, which coincides with $\text{Norm}\, (\Diag (\Gamma))/ \Cent(\Diag (\Gamma))$ (normalizer and centralizer). It can be identified with a subgroup of the symmetric group $\text{Sym}(S)$, but also with a subgroup of $\aut(U(\Gamma))$, since the bijection $\alpha_{\varphi}\colon S\to S$ can be extended to an automorphism of the universal grading group.
\smallskip

The term ``Weyl group'' is suitable, because  if $\Gamma$ is the Cartan grading on a semisimple
Lie algebra $\L$, which is produced by a maximal torus $T$ of $\aut(\L)$, then $\W(\Gamma)\cong \text{Norm}\, (T)/ T$ is isomorphic to the so-called \emph{extended Weyl group} of $\L$, i.e., the automorphism group of the root system of $\L$.

%%%%%%%%%%%%%%%%%%%%%%%%%%%%%%%%%%%%%%%%%%%%%%%%%%%%%%
%%%%%%%%%%%%%%%%%%%%%%%%%    SECTION 3  %%%%%%%%%%%%%%
%%%%%%%%%%%%%%%%%%%%%%%%%%%%%%%%%%%%%%%%%%%%%%%%%%%%%%

\section{Models}\label{sec_modelos3}

Here we review some constructions of the exceptional simple Lie algebra of type $E_6$, and add some other models of it which will allow us to have several different descriptions of equivalent gradings. They will be used in the last section to compute the symmetries of such gradings, changing the viewpoint when necessary.
Recall that we are assuming that $\F$ is an algebraically closed field of characteristic zero throughout the work, in spite of that most of the next constructions can be considered in more general settings.

\subsection{Involved structures}\label{subsec_involvedstructures}

Recall   that a \emph{composition algebra} is an $\F$-algebra $C$ endowed with a nondegenerate multiplicative quadratic form (the \emph{norm}) $n\colon  C \rightarrow \F$, that is,
$n(xy)=n(x)n(y)$ for all $x,y\in C$.
 Denote the associated bilinear form (the \emph{polar form})  by
$n(x,y):=n(x+y)-n(x)-n(y)$. %It is known that composition algebras can have only dimension 1, 2, 4 or 8.

The unital composition algebras are called \emph{Hurwitz algebras}. Each Hurwitz algebra satisfies a quadratic equation
$$
x^2 - t(x)x + n(x)1=0,
$$
where the linear map $t(x):=n(x,1)$ is called the \emph{trace} form. Besides it has a \emph{standard involution}   defined by
$\bar{x}:=t(x)1-x$. The subspace of  trace zero elements will be denoted by $C_0$.

Composition algebras can only have dimensions $1$, $2$, $4$ and $8$.
Over our algebraically closed field $\F$, the only composition algebras are,
up to isomorphism, the ground field $\F$, the cartesian product of two copies of the
ground field $\F\oplus \F$ (with the norm of $(a,b)$ given by the product $ab$), the matrix
algebra   $\Mat_{2\times2}(\F)$ (where the norm of
a matrix is its determinant), and the split \emph{octonion algebra}. (See, for instance,
\cite[Chapter 2]{libroRussi}, for all this material.)  In the following we will denote this  {octonion algebra} or \emph{Cayley algebra}
by $\C$.
 There is a \emph{canonical basis} $\{e_1, e_2, u_1, u_2, u_3, v_1, v_2, v_3\}$ of $\C$ consisting of isotropic elements, where $e_1$ and $e_2$ are idempotents, and such that $n(e_1,e_2)=n(u_i,v_i)=1$ and $n(e_k,u_i)=n(e_k,v_i)=n(u_i,u_j)=n(u_i,v_j)=n(v_i,v_j)=0$ for any $k=1,2$ and $1\leq i\neq j \leq 3$.
This basis is associated to a fine $\Z^2$-grading, called the \emph{Cartan grading} on $\C$, which is given by
\begin{equation}\label{eq_graddeCdeCartan}
\begin{array}{rll}
%\C_{(0,0)} & =  \F e_1 \oplus \F e_2,  \\
%\C_{(1,0)} & =  \F u_1, \quad  \C_{(0,1)} = \F u_2, \quad \C_{(1,1)} = \F v_3, \\
%\C_{(-1,0)} & = \F v_1, \quad  \C_{(0,-1)} = \F v_2, \quad \C_{(-1,-1)} = \F u_3.
%\Gamma_{\C}^1:
&\C_{(0,0)}  =  \mathbb{F}e_1 \oplus \mathbb{F}e_2,&\\
&\C_{(1,0)}  =   \mathbb{F}u_1, \qquad&  \C_{(-1,0)} = \mathbb{F}v_1,\\
&\C_{(0,1)}  =   \mathbb{F}u_2, &  \C_{(0,-1)} = \mathbb{F}v_2,\\
&\C_{(-1,-1)}  =   \mathbb{F}u_3, \qquad&  \C_{(1,1)} = \mathbb{F}v_3.
\end{array}
\end{equation}

Another grading on the octonion algebra arises quite naturally.
If we take any three orthogonal elements $w_1, w_2, w_3 \in \C$ such that $w_i^2=1$ for all $i$ in $\{1,2,3\}$, then $\C$ is spanned by
$ \{ 1, w_1, w_2, w_3, w_1w_2, w_2w_3, w_3w_1, w_1w_2w_3 \}, $
and this produces a fine $\Z_2^3$-grading on $\C$ by means of
\begin{equation}\label{eq_gradZ2cubodeC}
\deg(w_1) = (\bar1,\bar0,\bar0),\,\deg(w_2) = (\bar0,\bar1,\bar0),\,\deg(w_3) = (\bar0,\bar0,\bar1).
\end{equation}
Moreover the restriction to the $4$-dimensional composition algebra $\langle{1, w_1, w_2,   w_1w_2}\rangle$ is a $\Z_2^2$-grading and the restriction to the two-dimensional composition algebra $\langle{1, w_1}\rangle$ is a $\Z_2$-grading (the grading provided by the Cayley-Dickson doubling process applied to the field).

\smallskip

The algebra of derivations of $\C$ is a simple Lie algebra of type $G_2$, which   will be denoted by $\g_2$.
For any $a,b \in \C$, the endomorphism
$$ d_{a,b}=[l_a, l_b]+[l_a, r_b]+[r_a, r_b] $$
is a derivation of $\C$, %(called the \emph{inner derivation determined by $a$ and $b$}),
where $l_a(b)=ab=r_b(a)$ for any $a,b\in \C$.
The linear span $d_{\C,\C}=d_{\C_0,\C_0}=\{\sum_id_{x_i,y_i}\mid x_i,y_i\in\C_0,i\in \Z\}$ of these derivations coincides with the whole algebra $\der(\C)$ (\cite{Schafer}).
There is an algebraic group isomorphism $\aut(\C) \rightarrow \aut(\g_2)$ mapping each $\varphi\in \aut(\C)$ into the automorphism $\Ad(\varphi)$ defined by   $\Ad(\varphi)(d)=  \varphi d \varphi^{-1}$ for any $d\in \der(\C)$.
Therefore, $G$-gradings on $\C$ can be extended to $G$-gradings on $\der(\C)$.

\smallskip

Within our work we will often use a special class of composition algebras: the symmetric composition algebras (see \cite{ElduqueCompo1}).
\defi\rm A composition algebra $(S,\ast,n)$ (where $\ast$ denotes the product in $S$ and $n$ the norm) is said to be $\emph{symmetric}$ if the polar form of its norm is associative, that is:
$$ n(x\ast y,z)= n(x, y\ast z) $$
for any $x,y,z\in S$. \smallskip

If $(C,\cdot,n)$ is a Hurwitz algebra, then $pC=(C,*,n)$ with the new product $x*y:=\bar x\cdot\bar y$ is called a \emph{para-Hurwitz algebra}, which is an important example of symmetric composition algebra. Note that the unit of $(C,\cdot,n)$ becomes
a paraunit in $pC$, that is, an element $e$ such that $e * x = x * e = n(e, x)e - x$.
Within our setting, the only  symmetric composition algebras are one para-Hurwitz algebra of each dimension $1$, $2$, $4$ and $8$, and the pseudooctonion algebra, also of dimension $8$. This is defined as $P_8(\F)=(\Mat_{3\times3}(\F)_0,*,n)$, for the new product
$$
x * y = \omega xy - \omega^2 yx -\frac{\omega-\omega^2}3\tr(xy)I_3,
$$
with $\omega\in \F$ a primitive cubic root of the unit and the norm given by $n(x)=\frac16\tr(x^2)$.
Throughout the text, $I_n$ will denote the identity matrix of size $n$.

The gradings on  symmetric composition algebras are described and classified in \cite[Theorem~4.5]{Eldgradsencomposimetricas}.
Following that work,
every group grading on  a Hurwitz algebra $(C,\cdot,n)$  is a grading on $pC$, and the gradings coincide when  $C$ has dimension at least $4$.
There is a remarkable $\mathbb{Z}_3$-grading in the case of dimension $2$ which does not come from a grading on the corresponding Hurwitz algebra, namely,
$$
pC_{\bar{0}}=0,\qquad pC_{\bar{1}}=\F e_1\qquad  pC_{\bar{2}}=\F e_2,
$$
where $e_1$ and $e_2$ are the orthogonal idempotents $(1,0)$ and $(0,1)$ (idempotents for the product $\cdot$), which verify
$e_1*e_1=e_2$ and $e_2*e_2=e_1$.
Also, a natural  $\mathbb{Z}_3^2$-grading appears on the pseudooctonion algebra,    determined by
$$
P_8(\F)_{(\bar1,\bar0)}=\F \left( \begin{array}{ccc}
1 & 0 & 0\\ 0 & \omega & 0\\ 0 & 0 & \omega^2
\end{array} \right)                           ,\qquad
P_8(\F)_{(\bar0,\bar1)}=\F \left( \begin{array}{ccc}
0 & 1 & 0\\ 0 & 0 & 1\\ 1 & 0 & 0
\end{array} \right)   ,
$$
which is closely related to the $\mathbb{Z}_3$-gradings occurring on exceptional Lie algebras, as shown in \cite{Eldgradsencomposimetricas}.

\bigskip

Let $C$ be a Hurwitz algebra and consider the Jordan algebra of hermitian   matrices of size $3$,
$$ J=\mathcal{H}_3(C):=
\left\{ \left( \begin{array}{ccc}
\alpha_1 & \bar{a}_3 & a_2\\ a_3 & \alpha_2 & \bar{a}_1\\ \bar{a}_2 & a_1 & \alpha_3
\end{array} \right) \mid \alpha_1, \alpha_2, \alpha_3 \in \F, \;a_1, a_2, a_3 \in C \right\} ,$$
with commutative multiplication given by $xy:=\frac{1}{2}(x\cdot y + y \cdot x)$ (where $x \cdot y$ is the usual matrix product) and \emph{normalized trace form} given by $t_J=\frac{1}{3}(\alpha_1+\alpha_2+\alpha_3)$.
For such an algebra there is a decomposition $J=\F 1\oplus J_0$ for $J_0=\{x\in J\mid  t_{J}(x)=0\}$. For any $x,y\in J_0$, the product 
$$
x \ast y= xy - t_{J}(xy)1
$$
 gives a commutative multiplication on $J_0$. A remarkable Jordan algebra of this type is obtained starting from the Cayley algebra $\C$, that is, $\alb:=\mathcal{H}_3(\C)$, which is called the \emph{Albert algebra}, and
is the only exceptional simple Jordan algebra over $\F$ (see, for instance, \cite{Schafer}).

The algebra of derivations of $\alb$ is a simple Lie algebra of type $F_4$, which   will be denoted by $\f_4$.
As for $\der(\C)$, there is also an algebraic group isomorphism $\aut(\alb) \rightarrow \aut(\der(\alb))$ mapping each $\varphi\in \aut(\alb)$ to the automorphism $d\mapsto \varphi d \varphi^{-1}$, for any $d\in\der(\alb)$. Thus, gradings on $\alb$ are extended to gradings on $\f_4$.
Note that the linear map $D_{x,y}\colon  \alb \rightarrow \alb$ defined by
$$
D_{x,y}(z)=x(yz)-y(xz)
$$
for any $x,y,z\in\alb$, is a derivation and that the linear span $D_{\alb, \alb}=\{\sum_i D_{x_i,y_i}\mid x_i,y_i\in\alb\}$ of these derivations fills   the entire algebra $\der(\alb)=\f_4$.

\subsection{Tits' model}\label{subsec_Titsmodel}

In 1966, Tits provided a beautiful unified construction of all the exceptional simple Lie algebras
%$\f_4$, $\e_6$, $\e_7$ and $\e_8$
(\cite{Tits}). (Recall our assumptions on the field, although the construction is valid over fields of characteristic $\neq 2,3$.) When in this construction we use  a composition algebra and a
 Jordan algebra of hermitian $3\times 3$ matrices over a second composition algebra,
Freudenthal's Magic Square   is obtained:\smallskip

\begin{center}
 {\small
 \begin{tabular}{c|cccc}
 $\T(C,J)$ & $\mathcal{H}_3(\mathbb{F})$& $\mathcal{H}_3(\F\oplus \F)$& $\mathcal{H}_3(\Mat_2(\F))$& $\mathcal{H}_3(\C)$\\
\hline
 $\F$& $A_1$&$A_2$&$C_3$&${ F_4}$\\
 $ {\F\oplus \F}$& $A_2$&$A_2\oplus A_2$&$A_5$&${ E_6}$\\
 $ {\Mat_2(\F)}$& $C_3$&$A_5$&$D_6$&${ E_7}$\\
 $\C$& ${ F_4}$&${ E_6}$&${ E_7}$&${ E_8}$\\
 \end{tabular}}
 \end{center}  \medskip

\noindent This construction is reviewed here. For $C,C'$ two Hurwitz algebras and $J=\mathcal{H}_3(C')$, consider
$$
\T(C,J)=\der\,(C)\oplus (C_0 \otimes J_0) \oplus \der\,(J)
$$
with the anticommutative multiplication given by
$$
\begin{array}{l}
\bullet\  \der(C)\oplus   \der(J) \text{ is a Lie subalgebra of } \T(C,J), \\
\bullet\  {[}d, a\otimes x]=d(a) \otimes x, \\
\bullet\  [D, a\otimes x]=a \otimes D(x), \\
\bullet\  {[}a\otimes x, b\otimes y]= t_{J}(xy) d_{a,b}+[a,b]\otimes (x\ast y)+ 2t_C(ab)D_{x,y},
\end{array}
$$
for all $d\in \der(C)$, $D\in \der(J)$, $a,b\in C_0$ and $x,y\in J_0$. Then, $\T(C,J)$ is a Lie algebra. In particular, note that $\T(\C,\mathcal{H}_3(\F\oplus \F))$ and $\T(\F\oplus \F,\mathcal{H}_3(\C))$ (with $\C$   the Cayley algebra) are both Lie algebras of type $E_6$, and hence, isomorphic. We call   $\e_6:=\T(\F\oplus \F,\alb)$, identified with $\der(\alb) \oplus \alb_0$, which is naturally $\Z_2$-graded with even part $\f_4$.

Gradings on the  two ingredients involved in Tits' construction  (composition and Jordan algebras) can be used to get some interesting gradings on the resulting Lie  algebras.
Indeed, any automorphism $\varphi\in \aut (C)$ can be extended to an automorphism $\widetilde{\varphi}$ of $\T(C,J)$ by defining
$$
\begin{array}{l}
\widetilde{\varphi}(d) :=  \varphi d \varphi^{-1},  \\
\widetilde{\varphi}(a\otimes x)  :=  \varphi(a) \otimes x,\\
\widetilde{\varphi}(D)  :=  D,
\end{array}
$$
 for all $d\in \der(C)$,  $D\in \der\, (J)$, $a\in C_0$ and  $x\in J_0$.
 Similarly, any $\psi \in \aut(J)$ can be extended to an automorphism $\widetilde{\psi}$ of $\T(C,J)$.

\subsection{Elduque's model}\label{subsec_Elduquemodelo}

Even though Tits' construction is not
symmetric in the two composition algebras that are being used, the resultant Magic Square is indeed symmetric.
Thus,
some more
symmetric constructions of the exceptional Lie algebras
were developed by several authors.
Among them, one remarkable is that one in \cite{ElduqueCompo1}, based on symmetric composition algebras.

Let $(S,\ast, q)$ be a symmetric composition algebra and let
$$ \mathfrak{o}(S,q)=\{d\in \End_\F(S)\mid q(d(x),y)+q(x,d(y))=0\,\forall x,y\in S\} $$
be the corresponding orthogonal Lie algebra. On the subalgebra of $\mathfrak{o}(S,q)^3$ defined by
$$ \mathfrak{tri}(S,\ast,q)=\{(d_0,d_1,d_2)\in \mathfrak{o}(S,q)^3 \mid
d_0(x\ast y)=d_1(x)\ast y + x\ast d_2(y)\;\forall x,y\in S\}, $$
we have an order three automorphism $\vartheta$ given by
$$ \vartheta\colon  \mathfrak{tri}(S,\ast,q)\longrightarrow \mathfrak{tri}(S,\ast,q), \quad (d_0, d_1, d_2)\longmapsto (d_2, d_0, d_1), $$
which is called the \emph{triality automorphism}. Take the element of $\mathfrak{tri}(S,\ast,q)$ (denoted by $\mathfrak{tri}(S)$ when it is no ambiguity) given by
$$
t_{x,y}:=\left(\sigma_{x,y},\frac{1}{2}q(x,y)id-r_x l_y,\frac{1}{2}q(x,y)id-l_x r_y\right),
$$
where $\sigma_{x,y}(z)=q(x,z)y-q(y,z)x$, $r_x(z)=z\ast x$, and $l_x(z)=x\ast z$ for any $x,y,z\in S$.
Given two symmetric composition algebras $(S,\ast, q),(S',\star, q')$ we can consider the following symmetric construction:
$$
\mathfrak{g}(S,S') := \mathfrak{tri}(S,\ast, q)\oplus \mathfrak{tri}(S',\star, q')
\oplus (\bigoplus_{i=0}^2 \iota_i(S\otimes S'))
$$
where $\iota_i(S\otimes S')$ is just a copy of $S\otimes S'$ ($i=0,1,2$), and the skewsymmetric product is given by
$$
\begin{array}{l}
\bullet\  \mathfrak{tri}(S,\ast, q)\oplus \mathfrak{tri}(S',\star, q')  \text{ is a Lie subalgebra of }  \g(S,S'), \\
\bullet\   [(d_0,d_1,d_2), \iota_i(x\otimes x')]=\iota_i(d_i(x)\otimes x'), \\
 \bullet\   [(d'_0,d'_1,d'_2), \iota_i(x\otimes x')]=\iota_i(x\otimes d'_i(x')), \\
\bullet\   [\iota_i(x\otimes x'), \iota_{i+1}(y\otimes y')]=\iota_{i+2}((x*y)\otimes(x'\star y')), \\
\bullet\   [\iota_i(x\otimes x'),\iota_i(y\otimes y')] = q'(x',y')\vartheta^i(t_{x,y}) + q(x,y)\vartheta'^i(t'_{x',y'}), %\in\mathfrak{tri}(S)\oplus\mathfrak{tri}(S'),
\end{array}
$$
for any $(d_0,d_1, d_2)\in \mathfrak{tri}(S)$, $(d'_0,d'_1, d'_2)\in \mathfrak{tri}(S')$, $x,y\in S$ and $x',y'\in S'$,  being $\vartheta$ and $\vartheta'$ the corresponding triality automorphisms. The so defined algebra $\mathfrak{g}(S,S')$ is a Lie algebra (\cite[Theorem~3.1]{ElduqueCompo1}), and the Freudenthal's Magic Square   is obtained again. In particular, $\mathfrak{g}(S,S')$ is of type $E_6$ whenever
%$(\dim S,\dim S')=(2,8)$ or $(8,2)$.
one of the symmetric composition algebras involved has dimension $8$ and the other one has dimension $2$.

\smallskip

This construction is equipped with the following $\Z_2^2$-grading:
$$
\begin{array}{l}
 \g(S,S')_{(\bar{0},\bar{0})}=\mathfrak{tri}(S,\ast, q)\oplus \mathfrak{tri}(S',\star, q'), \\
 \g(S,S')_{(\bar{1},\bar{0})}=\iota_0(S\otimes S'), \\
  \g(S,S')_{(\bar{0},\bar{1})}=\iota_1(S\otimes S'), \\
   \g(S,S')_{(\bar{1},\bar{1})}=\iota_2(S\otimes S'),
\end{array}
$$
and also with the $\Z_3$-grading produced by the next extension of the triality automorphism, the map $\vartheta\in\aut(\g(S,S'))$ given by
$
\vartheta(d_0,d_1,d_2):=(d_2,d_0,d_1)$, $\vartheta(d'_0,d'_1,d'_2):=(d'_2,d'_0,d'_1)$ and $\vartheta(\iota_{i}(x\otimes x')):=\iota_{i+2}(x\otimes x'),
$
   for any $x\in S$, $x'\in S'$, $(d_0, d_1, d_2)\in \mathfrak{tri} (S)$, $(d'_0, d'_1, d'_2)\in \mathfrak{tri} (S')$.

Moreover,
if $f,f'$ are automorphisms of $S$ and $S'$ respectively, we can extend them to an automorphism $(f,f')$ of $\g(S,S')$ given by:
$$
\begin{array}{ll}
(f,f')(d_0,d_1,d_2)  &:=  (f d_0 f^{-1}, f d_1 f^{-1}, f d_2 f^{-1}), \\
(f,f')(d'_0,d'_1,d'_2)  &:=  (f' d'_0 f'^{-1}, f' d'_1 f'^{-1}, f' d'_2 f'^{-1}),\\
(f,f')(\iota_i(x\otimes x'))  &:=  \iota_i (f(x)\otimes f'(x')).
\end{array}
$$
This automorphism commutes with $\vartheta$
and with the automorphisms producing the $\Z_2^2$-grading.

\subsection{The 5-grading model}\label{subsec_modelo5grad}

We are interested in finding a construction of a Lie algebra of type $E_6$ in which an order two outer
automorphism fixing a subalgebra of type $C_4$ is easy to describe, as well as some fine gradings involving it.
The knowledge of such gradings (extracted from \cite{e6}) make us think that a good $\mathbb{Z}$-grading as a starting point would be the
related one to the following marked Dynkin diagram:

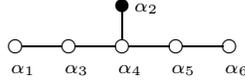
\setlength{\unitlength}{0.07in}
\begin{center}{\vbox{\begin{picture}(25,5)(4,-0.5)
\put(5,0){\circle{1}} \put(9,0){\circle{1}} \put(13,0){\circle{1}}
\put(17,0){\circle{1}} \put(21,0){\circle{1}}
\put(13,3){\circle*{1}}
\put(5.5,0){\line(1,0){3}}
\put(9.5,0){\line(1,0){3}} \put(13.5,0){\line(1,0){3}}
\put(17.5,0){\line(1,0){3}} \put(13,0.5){\line(0,1){2}}
\put(4.7,-2){$\scriptstyle \alpha_1$} \put(8.7,-2){$\scriptstyle \alpha_3$}
\put(12.7,-2){$\scriptstyle \alpha_4$} \put(16.7,-2){$\scriptstyle \alpha_5$}
\put(20.7,-2){$\scriptstyle \alpha_6$} \put(13.9,2.6){$\scriptstyle \alpha_2$}
\end{picture}}}\end{center}\medskip

Recall that the gradings on simple Lie algebras over cyclic groups, in particular the  $ \mathbb{Z}$-gradings,
were classified by Kac (\cite[Chapter VIII]{Kac}).

If $\Phi$ denotes the root system of $L $, a Lie algebra of type $E_6$, relative to a Cartan subalgebra $\mathfrak{h}$,
 $L_{\alpha}$ are the corresponding root spaces,
$\Delta=\{\alpha_1,\dots,\alpha_6\}$  is a basis of $\Phi$ and
$\sum_{i=1}^6n_i\alpha_i$ is the maximal root, then the above diagram determines a $ \mathbb{Z}$-grading
on $L$  as follows:
$$
L_p=\bigoplus_{\alpha\in\Phi\cup\{0\}}\{L_\alpha\mid
\alpha=\sum_{j=1}^6k_j\alpha_j,\, k_2=p\},
$$
for each $p\in\Z$, so that
$ L=\bigoplus_{p=-2}^2 L_p$.
The subalgebra $ L_0$ is reductive with
one-dimensional center $Z=\{h\in \mathfrak{h}\mid \alpha_j(h)=0\quad\forall
j\ne 2\}$ and semisimple part with root system $\{
\alpha\in\Phi\mid \alpha=\sum_{j\ne 2}k_j\alpha_j\}$ and basis
$\{\alpha_j\mid j\ne 2\} $, that is, an algebra of type $A_5$. Moreover,
\cite[p.\,108]{enci} implies that, in this $\mathbb{Z}$-grading, all the homogeneous components
(different from $L_0$) are $L_0$-irreducible modules. By counting the roots, one immediately observes that $\dim L_{\pm2}=1$ and $\dim L_{\pm1}=20$.

Therefore, if we take
  $V$   a $6$-dimensional vector space, we can identify $L_0$ with $ \mathfrak{gl}(V)$ (direct sum of a Lie algebra of type $A_5$ with a one-dimensional center) and we can also identify the components $L_2$,  $L_{1}$, $L_{-1}$ and $L_{-2}$
  with  $\bigwedge^6V$, $\bigwedge^3V$,  $\bigwedge^3V^*$ and $\bigwedge^6V^*$ respectively.
  Well known results of representation theory imply, for these $\slin_6(\F)$-modules, the key facts  $\dim\hom_{\slin(V)}(L_n\otimes L_m,L_{n+m})\le1$ if $n+m\ne0$, 
  $\dim\hom_{\slin(V)}(L_n\otimes L_{-n},\slin(V))=1$ and
  $\dim\hom_{\slin(V)}(L_n\otimes L_{-n},\F)=1$,
  for $n\in\{1,2\}$. As the restriction of the bracket
  $[\ ,\ ]\mid_{L_n\times L_m}\colon L_n\times L_m\to L_{n+m}$ is an $L_0$-invariant map for each $n,m\in\{\pm2,\pm1,0\}$, the above allows to recover the bracket once  we find  the $L_0$-invariant maps from
$L_n\times L_m$ to $L_{n+m}$.

These maps are standard in multilinear algebra (the reader can consult \cite[Appendix~B]{FulHar}).
First, for each natural $n$,
we have a bilinear product $\bigwedge^nV\times \bigwedge^nV^*\rightarrow \F$, $(u, f) \mapsto \langle u, f\rangle := \det(f_i(u_j))$,
if $u=u_1\wedge\dots\wedge u_n\in\bigwedge^nV$
and $f=f_1\wedge\dots\wedge f_n\in\bigwedge^nV^*$.
In particular we have $L_0$-invariant maps $L_{2}\times L_{-2}\to L_{0}$ and $L_{1}\times L_{-1}\to L_{0}$.
Second, we can consider  the contractions
$$
\begin{array}{lrcl}
L_1\times L_{-2}\rightarrow L_{-1},&(u,f)&\mapsto &u\lrcorner f,\\
L_2\times L_{-1}\rightarrow L_1,&(v,g) &\mapsto &v\llcorner g,
\end{array}
$$
determined  by $\langle w,u\lrcorner f\rangle=\langle w\wedge u,f\rangle$ and
$\langle  v\llcorner g,h\rangle=\langle v,g\wedge h\rangle$
for any $w\in L_1$
and $h\in L_{-1}$.
Third, the wedge product gives skewsymmetric maps $L_1\times L_1\to L_2$ and $L_{-1}\times L_{-1}\to L_{-2}$.
Finally, for each $n\in\{1,2\}$, $u\in L_n$ and $f\in L_{-n}$, define $\psi_{f,u}$ as the only element in $\slin(V)$ verifying $\tr(g \psi_{f,u})=\langle g\cdot u,f\rangle$ for any $g\in\slin(V)$ ($\cdot$ denotes the natural action of $\slin(V)$ on $\bigwedge^{3n}V$). This   provides too an $L_0$-invariant map $L_{-n}\times L_n\to L_{0}$.

The above guarantees the existence of scalars  $\alpha_{ij},\beta_i\in\F^\times$ such that the product in $L$ is
given by
\begin{equation}\label{eq_elproductoenlaZgrad}
\begin{array}{ll}
\bullet \ [u,v] = \alpha_{1,1 }\,u\wedge v ,  &\bullet \ [u',f'] = \alpha_{2,-2 }\,\psi_{f',u'} +\beta_2\,\langle u',f'\rangle  ,  \\
\bullet \ [u,f'] =\alpha_{1,-2 }\,u\lrcorner f'  , &\bullet \ [u',f]=\alpha_{-1,2 }\,u'\llcorner f  ,  \\
\bullet \ [u,f] = \alpha_{1,-1 }\,\psi_{f,u} +\beta_1\,\langle u,f\rangle  ,  &\bullet \ [f,g] = \alpha_{-1,-1 }\,f\wedge g ,
%\bullet \ [u',f'] = \alpha_{2,-2 }\psi_{f',u'} +\beta_2\langle u',f'\rangle  ,  \\
%\bullet \ [u',f]=\alpha_{-1,2 }u'\llcorner f  , \\
%\bullet \ [f,g] = \alpha_{-1,-1 }f\wedge g ,
\end{array}
\end{equation}
for any $u,v\in L_1$, $f,g\in L_{-1}$, $u'\in L_2$, $f'\in L_{-2}$,
 and the actions of $L_0$ on $L_n$ and $L_{-n}$ are the actions of the Lie algebra $\glin(V)$ on $\bigwedge^{3n}V$ and $\bigwedge^{3n}V^*$ respectively. These scalars can be determined by imposing the Jacobi identity, as in \cite{miomodelos}, but it is not necessary for our purposes.

Denote by $T_1$ the torus of $\aut(L)$ producing this $\mathbb{Z}$-grading,
that is, the group of automorphisms of the form $f_\alpha$ for $\alpha\in\F^\times$ such that $f_\alpha\vert_{L_n}=\alpha^n\id$ for all $n\in\Z$.
\smallskip

Fix $\ii\in \F$ with $\ii^2=-1$. Take a basis $\{e_i\mid i=1,\dots,6\}$ of $V$ and let $\{e_i^*\mid i=1,\dots,6\}\subset V^*$ be the dual basis. Define   $\theta\colon L\to L$  by:
$$
\begin{array}{l}
    \theta(e_{\sigma(1)}\wedge e_{\sigma(2)}\wedge e_{\sigma(3)})
   := \text{sg}(\sigma)\, \ii e_{\sigma(4)}\wedge e_{\sigma(5)}\wedge e_{\sigma(6)}, \\
   \theta(e^*_{\sigma(1)}\wedge e^*_{\sigma(2)}\wedge e^*_{\sigma(3)})
   := -\text{sg}(\sigma) \,\ii e^*_{\sigma(4)}\wedge e^*_{\sigma(5)}\wedge e^*_{\sigma(6)}, \\
    \theta(\alpha I_6+A):=\alpha I_6-A^t,   \\
  \theta|_{L_2\oplus L_{-2}} :=-\id,
\end{array}
$$
for any $\alpha\in \F, A\in\mathfrak{sl}(V),\sigma\in S_6$.
Making use of Equation~(\ref{eq_elproductoenlaZgrad}), some straightforward computations prove  that $\theta\in\text{Aut}(L)$ and that besides
 $\dim(\fix \theta)=36$. Hence $\theta$ is an outer order two automorphism fixing a subalgebra of type $C_4$ and commuting  with the torus $T_1$.

\medskip

For each $\varphi\in\text{GL}(V)$, let   $\widetilde{\varphi}\colon L\to L$ be
$$
\begin{array}{ll}
\widetilde{\varphi}(u) := \varphi(u_1)\wedge\dotsc\wedge \varphi(u_r),    \\
\widetilde{\varphi}(f) :=(\varphi\cdot f_1) \wedge\dotsc\wedge (\varphi\cdot f_r),    \\
\widetilde{\varphi}(F) :=\varphi F\varphi^{-1},
\end{array}
$$
for any $u=u_1\wedge\dotsc\wedge u_r\in L_1 \cup L_2$, $f= f_1 \wedge\dotsc\wedge f_r \in L_{-1}\cup L_{-2}$
and $F\in L_0$,
with $\varphi\cdot f_i=f_i\varphi^{-1}\in V^*$. Actually $\psi_{\widetilde{\varphi}(f),\widetilde{\varphi}(u)}=\widetilde{\varphi}(\psi_{f,u})$,
what implies that $\widetilde{\varphi} $ is an automorphism of $L$. With abuse of notation, we will consider $\text{GL}_6(\F)\leq\text{Aut}(\mathfrak{e}_6)$. Notice that these automorphisms also commute with the torus $T_1$ (the torus itself is contained in $\GL(V)$, since $f_\alpha=\alpha \widetilde{I}_6$).
Furthermore, the centralizer of $T_1$ in $\aut(L)$ equals $\span{\widetilde{\varphi}\mid \varphi\in\GL(V)}\rtimes \span{\theta}$.
Indeed, if $ \gamma\in\aut(L)$ commutes with $T_1$, the restriction $\gamma\vert_{\slin(V)}$ belongs to $\aut(\slin(V))$, which is well known to be isomorphic to $\text{PSL}\,(V)\rtimes\span{\theta\vert_{\slin(V)}}$.

\subsection{Adams' model}\label{subsec_modeloAdams}
\label{Adams}

 Following the lines of the decomposition   of $\g_2=\slin(W)\oplus W \oplus W^*$  (\cite[\S22.2 and \S22.4]{FulHar}), for $W$
a three-dimensional vector space, Adams gave in \cite[Chapter~13]{Adams} a model for $E_6$ using three copies of a $3$-dimensional vector space $W_1=W_2=W_3=W$ and their dual spaces $W_i^*$ as follows:
 take
 $\mathcal{L} = \L_{\bar0}\oplus \L_{\bar1}\oplus \L_{\bar2}$, for
$$
\L_{\bar0} = \slin(W_1)\oplus \slin(W_2) \oplus \slin(W_3),
\quad \L_{\bar1} = W_1 \otimes W_2 \otimes W_3,
\quad \L_{\bar2} = W_1^* \otimes W_2^* \otimes W_3^*,
$$
  where $\sum \slin(W_i)$ is a Lie
subalgebra of type $3A_2$, its actions on $ W_1\otimes W_2\otimes
W_3$ and $W_1^*\otimes W_2^*\otimes W_3^*$  are the natural ones
(the $i$th simple ideal acts on the $i$th slot), and
$$
\begin{array}{ll}
\bullet \  [\otimes f_i,\otimes u_i]&=\sum_{\buildrel{k=1,2,3}\over{i\ne j\ne k}}
 f_i(u_i)f_j(u_j)\big(f_k(-)u_k-\frac13f(u_k)\id_{W_k}\big),\\
\bullet \  [ \otimes u_i,\otimes v_i]&=\otimes (u_i\wedge v_i),\\
\bullet \  [ \otimes f_i,\otimes g_i]&=\otimes (f_i\wedge g_i),\nonumber
\end{array}
$$
 for any   $u_i,v_i\in W_i$, $f_i,g_i\in W^*_i$. We have fixed nonzero
   trilinear
alternating  maps $\det_i\colon W_i\times W_i\times W_i\to \F$, so
that $u_i\wedge v_i$ denotes the element in $W_i^*$ given by
$\det_i(u_i,v_i,-)$  and $f_i\wedge g_i$ denotes the element in
$W_i^*$ given by $\det_i^*(f_i,g_i,-)$, being $\det_i^*$ the dual
map of $\det_i$.
 It turns out that  $\L$ is a Lie algebra   of type $E_6$.

Denote by $H_1$ the automorphism producing the $\Z_3$-grading, that is, $H_1|_{\L_{\bar i}}=\omega^i \id$, with $\omega$ a primitive cubic root of $1$ in $\F$. As any automorphism is  determined by  its action on $\L_{\bar1}$,
we can take  $H_2$ the   automorphism   such that $H_2(u\otimes v\otimes w)=v\otimes w\otimes u$ for any $u,v,w\in W$. Furthermore, each $f\in \GL(W)$ extends to $\Psi(f) \in \aut(\L)$,   determined by   $\Psi(f)(u\otimes v\otimes w)=f(u)\otimes f(v)\otimes f(w)$. Note that $ \Psi\colon\GL(W) \rightarrow \aut(\L)$ is a homomorphism of algebraic groups with kernel $\{\omega^nI_3\mid n=0,1,2\}$.

In particular, if we identify any endomorphism of $W$ with its matrix relative to a fixed basis of $W$, and take %$p_{\alpha,\beta}=\text{diag}(\alpha,\beta,(\alpha\beta)^{-1})\in\GL_3(\F) $ and  \margen{quiz\'{a} quite los p}
$T_{\alpha,\beta}=\Psi(\text{diag}(\alpha,\beta,(\alpha\beta)^{-1}))\in\Psi(\GL_3(\F) )$, we get the quasitorus $ \langle H_1,H_2,T_{\alpha,\beta}\mid \alpha,\beta\in\F^\times\rangle$ ($\mathcal{Q}_2$ in \cite{e6}), that produces a fine $\Z_3^2\times\Z^2$-grading on $\L$.

%%%%%%%%%%%%%%%%%%%%%%%%%%%%%%%%%%%%%%%%%%%%%%%%%%%%%%

\subsection{ Model based on $\slin(2)\oplus\mathfrak{sl}(6)$}\label{subsec_modeloa5masa1}

Recall (see again \cite[Chapter~VIII]{Kac}) that there is a $\Z_2$-grading on any algebra of type $E_6$ with zero homogeneous component isomorphic to $A_5+A_1$ and the other one an irreducible module. Thus allows to endow     $\mathfrak{L}=\mathfrak{L}_{\bar0}\oplus \mathfrak{L}_{\bar1}$ with a Lie algebra structure, starting from
$U$ and $V$ vector spaces of dimensions $2$ and $6$ respectively, and by considering
   $\mathfrak{L}_{\bar0}= \slin(U)\oplus\slin(V)$ and $\mathfrak{L}_{\bar1}= U\otimes\bigwedge^3V$.

  For that purpose, take $b\colon U\times U\to\F$ a  (nonzero) skewsymmetric bilinear form, and note that $\mathfrak{sp}(U)=\slin(U)$.
  For each $v,w\in U$, consider the map $\varphi_{v,w}:=b(v,-)w+b(w,-)v\in\sp(U)$.
 If  $x=x_1\wedge x_2\wedge x_3\in\bigwedge^3V$ and $f\in\text{End}\,(V)$, denote  by
 $$
 \begin{array}{l}
 f(x):=f(x_1)\wedge x_2 \wedge  x_3 + x_1 \wedge f(x_2)\wedge  x_3 + x_1 \wedge  x_2 \wedge f(x_3),\\
 f\cdot x:=f(x_1)\wedge f(x_2) \wedge  f(x_3).
  \end{array}
  $$
  Fix a nonzero map $\bigwedge^6V\rightarrow \F$ and let $\langle\cdot,\cdot\rangle\colon \bigwedge^3V\times\bigwedge^3V\rightarrow \F$ be the related  (skewsymmetric) product. Note that $\langle h(x),y\rangle+\langle x,h(y)\rangle=(\text{tr}\, h)  \langle x,y\rangle$ and $\langle h\cdot x,h\cdot y\rangle=(\det h)  \langle x,y\rangle$
     for any $x,y\in\bigwedge^3V$ and any endomorphism $h\in\End(V)$.
 For each $x,y\in\bigwedge^3V$ we   denote by $[x,y]\,(=[y,x])$ the   element of $\slin(V)$ characterized by the property $\text{tr}(f[x,y])=\langle f(x),y\rangle$ for all $f\in\slin(V)$ (denote the composition by juxtaposition).

 Now we note, by using arguments as in Subsection~\ref{subsec_modelo5grad}, that there exist suitable scalars $\lambda,\mu\in \F^\times$ such that the  product on $ \mathfrak{L}$  given by
$$
\begin{array}{l}
\bullet \ \slin(U)\oplus\slin(V) \text{ is a subalgebra of   $\mathfrak{L}$}, \\
\bullet \ [g,v\otimes x]  =  g(v)\otimes x, \\
\bullet \ [f,v\otimes x]  =  v\otimes f(x), \\
\bullet \ [v\otimes x, w\otimes y]  =  \lambda\langle x,y\rangle \varphi_{v,w}  + \mu b(v,w)[x,y],
\end{array}
$$
for all $g\in\slin(U)$, $f\in\slin(V)$, $v,w\in U$ and $x,y\in\bigwedge^3V$, turns $\mathfrak{L}$ into a Lie algebra of type $E_6$.   %

The obvious advantage of this model is that allows to have a source of automorphisms commuting with that one producing the $\Z_2$-grading.
If we take $f_1\in \Sp(U)$ and $f_2\in \SL(V)$,
then the map $f_1\times f_2\colon \mathfrak{L}\to \mathfrak{L}$ given by
$$
\begin{array}{l}
  f_1\times f_2(g+f):= f_1gf_1^{-1}+f_2ff_2^{-1},\\
  f_1\times f_2(v\otimes x):= f_1(v)\otimes f_2\cdot x,
 \end{array}
   $$
   for any $g\in\slin(U)$, $f\in\slin(V)$, $v\in U$ and $x\in\bigwedge^3V$, is an automorphism of $\mathfrak{L}$.
    It is not difficult to prove it:
 $$
 \begin{array}{c}
 \text{tr}(f  (f_2[x,y]f_2^{-1}))= \text{tr}((f_2^{-1}ff_2) [x,y])=\langle f_2^{-1}ff_2(x),y\rangle\\
 =\langle f_2\cdot (f_2^{-1}ff_2)(x),f_2\cdot y\rangle
 = \langle f(f_2\cdot x),f_2\cdot y\rangle=\text{tr}(f [f_2\cdot x,f_2\cdot y]),
 \end{array}
 $$
 so that  $[f_2\cdot x,f_2\cdot y]=f_2[x,y]f_2^{-1}$. Besides $\langle f_2\cdot x,f_2\cdot y\rangle=\span{x,y}$,
 $ \varphi_{f_1(v),f_1(w)} =f_1 \varphi_{v,w} f_1^{-1}$, so that
 $f_1\times f_2 ([v\otimes x, w\otimes y])  =[f_1\times f_2 (v\otimes x),  f_1\times f_2( w\otimes y)]$.
 Analogously the rest of relations can be checked.
  Thus, as $f_1\times\id_V$ always commutes with $\id_U\times f_2$,   the direct product $\Sp(U)\times \SL(V)$ can be taken as a subalgebra of $\aut (\mathfrak{L})$.

%%%%%%%%%%%%%%%%%%%%%%%%%%%%%%%%%%%%%%%%%%%%%%%%%%%%%%

\subsection{ Model based on $\mathfrak{sp}_8(\F)$}\label{subsec_modeloc4}

In \cite[Section~5.3]{e6} a model of $\frak{e}_6$ based on an algebra of type $C_4$ is outlined.
Take the  vector space $\mathcal{V}=\F^8$ and the nondegenerate symplectic bilinear form
$b\colon \mathcal{V}\times \mathcal{V}\to\F$ given by $b(u,v)=u^tCv$ for the matrix
$$
C=
\left(\begin{array}{cccc}
0&\sigma_3&0&0\\
\sigma_3&0&0&0\\
0&0&0&\sigma_3\\
0&0&\sigma_3&0
\end{array}\right)
\qquad\text{with }
\sigma_3=\left(\begin{array}{cc}
0&-1\\
1&0
\end{array}\right).
$$
Let $\mathbf{L}_0:=\frak{sp}(\mathcal{V},b)$, which relative to the canonical basis coincides with
$\{x\in\Mat_{8\times8}(\F)\mid xC+Cx^t=0\}$,
a simple Lie algebra of type
$C_4$.
Now, we consider the contraction
$$
\begin{array}{llcl}
c\colon  &\bigwedge^4\mathcal{V}& \longrightarrow &\bigwedge^2\mathcal{V} \\
&v_1\wedge v_2\wedge v_3 \wedge v_4  &\mapsto&\sum_{\tiny
{\begin{array}{l}\sigma\in S_4\\
\sigma(1)<\sigma(2)\\
\sigma(3)<\sigma(4)\end{array}}}
(-1)^\sigma b(v_{\sigma(1)},v_{\sigma(2)})
v_{\sigma(3)}\wedge v_{\sigma(4)}.
\end{array}
$$
Its kernel $\ker c$
is isomorphic to $V(\lambda_4)+V(0)$ as $\mathbf{L}_0$-module ($\lambda_i$ the fundamental weights, with the notations in \cite{Humphreysalg}). Let $\mathbf{L}_1$ be the submodule of $\ker c$
  isomorphic to $V(\lambda_4)$.
Then the vector space $\mathbf{L}=\mathbf{L}_0\oplus \mathbf{L}_1$ can be endowed with a $\Z_2$-graded Lie algebra structure such that $\mathbf{L}$ turns out to be an algebra of type $E_6$.
Call  $ \Theta$ the grading automorphism.

For each  $f\in\Sp(\mathcal{V})$, the map $f^\diamondsuit\colon \mathbf{L}\to\mathbf{L}$ given by
$$
\begin{array}{l}
f^\diamondsuit(g):=f^{-1}gf, \\
f^\diamondsuit(v):=\sum_{ }f(v_{i_1})\wedge f(v_{i_2})\wedge f(v_{i_3}) \wedge f(v_{i_4}) ,
\end{array}
$$
 if $g\in  {\mathbf{L}}_{\bar0}$, $v=\sum_{ } v_{i_1}\wedge v_{i_2}\wedge v_{i_3} \wedge v_{i_4} \in {\mathbf{L}}_{\bar1} $,
  is an automorphism of $\mathbf{L}$. Furthermore,
  $\Cent_{\aut (\mathbf{L})}(  \Theta)=\{f^\diamondsuit\mid f\in\Sp(V)\}\cdot\span{ \Theta}\cong\Sp(V)\times\Z_2$.

%%%%%%%%%%%%%%%%%%%%%%%%%%%%%%%%%%%%%%%%%%%%%%%%%%%%%%
%%%%%%%%%%%%%%%%%%%%%%%%%    SECTION 4  %%%%%%%%%%%%%%
%%%%%%%%%%%%%%%%%%%%%%%%%%%%%%%%%%%%%%%%%%%%%%%%%%%%%%

\section{Description of the gradings}\label{sec_graduaciones4}

The main theorem  in \cite{e6} gives the classification up to equivalence of the fine gradings on $\e_6$. It encloses the following table, which provides, for    each fine grading, its  universal grading group and the type of the  grading. These invariants are enough to distinguish the considered equivalence class.
\begin{center}{
\begin{tabular}{ |c|c|}
\hline
 Universal grading group & Type \cr
\hline\hline
   $ \Z_3^4$   &  $( 72,0,2 )$ \cr
\hline
   $\Z^2\times\Z_3^2$   &  $( 60,9 )$  \cr
\hline
  $\Z_3^2\times\Z_2^3$   &  $( 64,7 )$  \cr
\hline
  $\Z^2\times\Z_2^3$   &  $( 48,1,0,7 )$  \cr
\hline
  $\Z^6 $   &  $(72,0,0,0,0,1  )$  \cr
\hline
  $\Z^4\times\Z_2$   &  $(72,1,0,1  )$  \cr
\hline
  $ \Z_2^6$   &  $( 48,1,0,7 )$ \cr
\hline
  $\Z\times\Z_2^4$   &  $(  57,0,7)$\cr
\hline
   $\Z_3^3\times\Z_2$   &  $( 26,26 )$  \cr
\hline
  $\Z^2\times\Z_2^3$   &  $(60,7,0,1   )$  \cr
\hline
  $\Z_4\times\Z_2^4$   &  $(48,13,0,1     )$  \cr
\hline
  $\Z \times\Z_2^5$   &  $( 73,0,0,0,1  )$  \cr
\hline
  $ \Z_2^7$   &  $( 72,0,0,0,0,1  )$  \cr
\hline
  $ \Z_4^3$   &  $( 48,15 )$  \cr
\hline
 \end{tabular}}\end{center}\smallskip

 In this section we describe the six gradings by infinite universal grading groups different from the Cartan grading. Our descriptions will not be, in most  cases, the same as  those ones in \cite{e6}, to adapt them to our study of symmetries.
This   makes necessary to recognize properties of the quasitori producing them, and, particularly, of the automorphisms involved.
There are five  conjugacy classes of order three automorphisms of $\e_6$, characterized by the isomorphism type of the fixed subalgebra. It is said that   such an automorphism is of type
$3B$, $3C$, $3D$, $3E$ and $3F$ when the fixed subalgebra is of type $A_5+Z$, $3A_2$, $D_4+2Z$, $A_4+A_1+Z$
and $D_5+Z$ respectively. Observe that in this case the dimension of the fixed subalgebra ($36$, $24$, $30$, $28$ and $46$, respectively) determines the conjugacy class.
Also, there are four conjugacy classes of order two automorphisms of $\e_6$. The inner ones have fixed subalgebras isomorphic to $A_5+ A_1$ and $D_5+Z$, and are called of type $2A$ and $2B$ respectively, and the outer ones have fixed subalgebras isomorphic to $F_4$ and $C_4$, and are called of type $2C$ and $2D$ respectively.

\subsection{Inner $  \Z_2^3\times\Z^2$-grading}\label{subsec_gradZcuadZ2cubo}

Consider the Tits'   construction $\T(\C,\J)=\der(\C) \oplus (\C_0 \otimes \J_0) \oplus\, \der(\J) $ for  $\J=\mathcal{H}_3(\F \oplus \F)$. Note that this Jordan algebra is isomorphic to $\Mat_{3\times3}(\F)^+$ with the symmetrized product.
Consider the $\Z^2$-grading on $\J$ given by
$$
\begin{array}{c}\J_{0,0} = \langle E_1, E_2, E_3\rangle, \\
 \J_{1,0}=\langle E_{12}\rangle, \; \J_{0,1}=\langle E_{23}\rangle, \; \J_{1,1}=\langle E_{13}\rangle, \\
  \J_{-1,0} = \langle E_{21}\rangle, \; \J_{0,-1}=\langle E_{32}\rangle, \; \J_{-1,-1}=\langle E_{31}\rangle,
  \end{array}
  $$
where $E_{ij}$ is the matrix with $1$ in the entry $(i,j)$ and $0$ elsewhere, and $E_i:=E_{ii}$.
Let $T_2:=\{s_{\alpha, \beta}\mid \alpha,\beta\in \F^\times\}\le\aut(\J)$ be the torus producing the grading, that is,
$s_{\alpha,\beta}\vert_{\J_{n,m}}= \alpha^n \beta^m\id$. We also use the notation $s_{\alpha, \beta}$ for its extension  to an automorphism  of $\T(\C,\J)$, so as $T_2$ as a torus of $\aut(\T(\C,\J))$.

Besides, we have a $\Z_2^3$-grading on $\T(\C,\J) $ induced by the grading  on $\C$
given by Equation~(\ref{eq_graddeCdeCartan}).
We will denote  by $f_1$, $f_2$ and $f_3$ the automorphisms of $\T(\C,\J)$ producing that grading.
As seen in \cite{e6}, a nontoral group isomorphic to $\mathbb{Z}_2^3$ is unique up to conjugation, and its centralizer is isomorphic to itself direct product with  a copy of the group $\PSL_3(\F)$ (in our case, the extension of $\aut(\J)$),
what proves that the quasitorus $ P_1:= T_2\langle  f_1, f_2, f_3 \rangle \cong (\F^\times)^2\times\Z_2^3$ is a MAD-group (conjugated to $\mathcal{Q}_4$ of \cite{e6}) producing a fine $\Z^2\times\Z_2^3$-grading   denoted by $\Gamma_1$.

\subsection{$ \Z_2^4\times\Z $-grading}\label{subsec_gradZZ2cuarta}

A generalization of the $\mathbb{Z}$-grading on the Albert algebra described in \cite[Subsection~4.1]{EK11}
is the following:
For $1$ and $1'$     paraunits of two para-Hurwitz algebras, $S$ and $S'$,
consider the map $d:=\ad(\iota_0(1\otimes 1'))$,  which is a semisimple derivation of $\g(S,S')$ with set of eigenvalues
$\{\pm2,\pm1,0\}$. Thus, the eigenspace decomposition gives the following
$\Z$-grading (5-grading) on $  \g(S,S')$:
\begin{equation}\label{eq_laZgrad}
 \begin{array}{l}
 \g(S,S')_{\pm2}=S_{\pm}(S_0,S'_0), \\
  \g(S,S')_{\pm1}=\nu_{\pm}(S\otimes S'), \\
   \g(S,S')_0 = t_{S_0, S_0}\oplus t_{S'_0,S'_0}\oplus \iota_0 (S_0\otimes S'_0) \oplus \F\iota_0 (1\otimes 1'),
 \end{array}
\end{equation}
where   $S_0$ and $S_0'$ denote the trace zero elements of $S$ and $S'$,
and we are using the notation
$$
\begin{array}{l}
S_{\pm}(y,y'):=t_{1,y}+t'_{1,y'}\pm \ii \iota_0(y\otimes 1 + 1\otimes y') \\
\nu_{\pm}(y\otimes y'):=\iota_1(y\otimes y')\mp \ii \iota_2(\bar{y}\otimes \bar{y}')
\end{array}
$$
 for all $y\in S$, $y'\in S'$, and for a fixed scalar $\ii \in \F$ such that  $\ii^2=-1$.
 Moreover, this $\mathbb{Z}$-grading can be refined with any grading coming from $S$ or $S'$, since the derivation $d$ commutes with the automorphism $(f,f')$ described in Subsection~\ref{subsec_Elduquemodelo}.

We particularize   this $\Z$-grading to  the algebra of type  $ E_6$ constructed from the para-Hurwitz algebras $S_8$ and $S_2$ of dimensions $8$ and $2$ respectively. Take  the $\Z_2^3$-grading on $S_8$
 and the $\Z_2$-grading on $S_2$ coming from the corresponding $\Z_2^3$ and $\Z_2$-gradings on $\C$ and $\F\oplus\F$ as in Subsection~\ref{subsec_involvedstructures}. 
Thus we obtain a $\Z_2^4\times\Z$-grading on $\g(S_8,S_2)$, denoted by $\Gamma_2$. This is a fine grading equivalent to that one produced by $\mathcal{Q}_8$ in \cite{e6},
because if $\varrho$ denotes the order two automorphism of $\g(S_8,S_2)$ producing  the $\Z_2$-grading which comes from extending that one on $S_2$, 
then $\varrho$ fixes a subalgebra of type $F_4$. Indeed, if we denote  $e_1=(1,0)$ and $e_2=(0,1)$, the grading on $p(\F\oplus\F)=S_2$ is given by $(S_2)_{\bar0}=\span{1=e_1+e_2}$ and $(S_2)_{\bar1}=\span{e_1-e_2}$, so that $\varrho
(\iota_i(x\otimes e_1))=\iota_i(x\otimes e_2)$ and $\fix\varrho=\mathfrak{tri}\,(S_8)\oplus(\oplus_{i=0}^2\iota_i(S_8\otimes \F 1))\cong\g(S_8,\F)\cong\f_4$.

\subsection{$ \Z_2\times\Z^4 $-grading}\label{subsec_gradZcuartaZ2}

First we provide  a $\Z^4$-grading on $ \g(S_8,S)$ inspired in \cite{EK11}, where $S$ and $S_8$ are  para-Hurwitz algebras, the second one of dimension $8$. Take $a_1=(1,0,0,0)$, $a_2=(0,1,0,0)$, $g_1=(0,0,1,0)$, $g_2=(0,0,0,1)$   generators of the group $\Z^4$, and $a_0=-a_1-a_2, g_0=-g_1-g_2$. Set the degrees of the $\Z^4$-grading as follows:
$$
\begin{array}{rcccl}
\deg\, \iota_i(e_1\otimes s) & = & a_i & = & -\deg\, \iota_i(e_2 \otimes s), \\
\deg\, \iota_i(u_i\otimes s) & = & g_i & = & -\deg\, \iota_i(v_i \otimes s), \\
\deg\, \iota_i(u_{i+1}\otimes s) & = & a_{i+2}+g_{i+1} & = & -\deg\, \iota_i(v_{i+1} \otimes s), \\
\deg\, \iota_i(u_{i+2}\otimes s) & = & - a_{i+1}+g_{i+2} & = & -\deg\, \iota_i(v_{i+2} \otimes s),
\end{array}
$$
where $s\in S$ and $B=\{e_1, e_2, u_0, u_1, u_2, v_0, v_1, v_2\}$ is a canonical basis of the algebra $S_8$. Also set $\deg(t)=(0,0,0,0)$ for all $t\in \mathfrak{tri}(S)$, and $\deg  t_{x,y}=\deg  \iota_0(x\otimes s) + \deg \iota_0(y\otimes s)$, if $x,y\in B$.
It is a straightforward computation that this assignment provides a $\mathbb{Z}^4$-grading on $\g(S_8,S)$.

In particular, all the exceptional Lie algebras in the fourth row of the magic square, namely,
$F_4$,  $E_6$, $E_7$ and $E_8$, are  $\mathbb{Z}^4$-graded.
In fact, such grading is a root grading for the root system of $F_4$, taking into account the interesting relationship between fine gradings and root gradings which has been recently stated in
\cite{ElduqueRootgradings}.

When applied to $E_6$, viewed as $\g(S_8, S_2)$ for $S_2$ the symmetric composition algebra $p(\F\oplus\F)$,
the torus $T_4$ inducing the $\Z^4$-grading commutes with the outer automorphism $\varrho$ and the quasitorus $P_3:=\langle T_4\cup\{\varrho\}\rangle$ provides a $\Z^4 \times \Z_2$-fine grading on $\g(S_8, S_2)$, denoted by $\Gamma_3$. This MAD-group is essentially $\mathcal{Q}_6$ of \cite{e6}, taking into account that the
 $\Z^4  $-grading produced by the restriction of $T_4$ to $\fix\varrho\cong\f_4$ is the Cartan grading.

%%%%%%%%%%%%%%%%%%%%%%%%%%%%%%%%%%%%%%%%%%%%%%%%%%%%%%%%%%%%%%%%%%%%%%%%%%%%%%%%%%%%%%%%%%%%%%%%%%%%

\subsection{$ \Z_3^2 \times\Z^2$-grading}\label{subsec_gradZcuadZ3cuad}

The $ \Z_3^2 \times\Z^2$-grading described in Subsection~\ref{subsec_modeloAdams}, produced by $\mathcal{Q}_2$ in \cite{e6}, admits different realizations convenient for our objectives.

Consider again the Tits' construction $\T(\C,\J)=\der(\C) \oplus (\C_0 \otimes \J_0) \oplus \der(\J) $ applied to the Jordan algebra  $\J=\mathcal{H}_3(\F \oplus \F)\cong\Mat_{3\times3}(\F)^+$. Its Lie algebra of derivations $\der(\J)$ can be identified with $\J_0$ as a vector space, since for $x\in\J_0$, the map $\ad(x)$ is a derivation of $\Mat_{3\times3}(\F)^+$,  being $\ad(x)(y)=xy-yx$.
This allows to identify $\e_6$ with $ \der(\C) \oplus (\C_0 \otimes \J_0) \oplus \F \otimes \J_0$ and then, as
$\C=\C_0\oplus\F$,  with
%\begin{equation*}
\begin{equation}\label{eq_paraverbienlosmodulos}
\e_6\cong\der(\C) \oplus (\C \otimes \J_0).
\end{equation}

 Think of the $\Z_3^2$-grading obtained on $\e_6$ when grading $\J$ through the assignment $\deg(b)=(\bar1,\bar0)$, $\deg(c)=(\bar0,\bar1)$, for
$$\small{
b=\left( \begin{array}{ccc}
1 & 0 & 0\\ 0 & \omega & 0\\ 0 & 0 & \omega^2
\end{array} \right), \quad
c=\left( \begin{array}{ccc}
0 & 0 & 1\\ 1 & 0 & 0\\ 0 & 1 & 0
\end{array} \right)},
$$
where $\omega$ denotes a primitive cubic root of the unit, and then take $h_1$ and $h_2$ as the automorphisms of $\e_6$ producing such  $\Z_3^2$-grading.  Observe that the simultaneous diagonalization has   fixed subalgebra isomorphic to $\g_2 = \der (\C)$  and all the other homogeneous components  can be identified with   $\C$  as $L_e$-modules (obvious from Equation~(\ref{eq_paraverbienlosmodulos})), so that every order three element in
the subgroup $\span{h_1,h_2}$   has type $3D$.

Moreover, consider the $\Z^2$-grading   on $\C$, which is extended to a $\Z^2$-grading on $\e_6$. The two-dimensional torus producing this grading will be denoted by $\T_2$. Then, the quasitorus $ P_4:=\langle h_1, h_2, \T_2 \rangle \cong\Z_3^2 \times (\F^\times)^2 $ is conjugated to $\mathcal{Q}_2$ (although $\span{h_1,h_2}$ is not conjugated to $\span{H_1,H_2}$, as proved in \cite{e6}), so that it produces a fine $\Z_3^2\times \Z^2 $-grading of type $(60,9)$ on $\e_6$, denoted by $\Gamma_4$.

%%%%%%%%%%%%%%%%%%%%%%%%%%%%%%%%%%%%%%%%%%%%%%%%%%%%%%%%%%%%%%%%%%%%%%%%%%%%%%%%%%%%%%%%%%%%%%%%%%%%

\subsection{$ \Z_2^5 \times\Z$-grading}\label{subsec_gradZZ2quinta}

For this grading, the 5-grading model in Subsection~\ref{subsec_modelo5grad} is used.
Note that, if $\varphi\in \GL(V)$, the condition $\theta\widetilde{\varphi}(A)=\widetilde{\varphi}\theta(A)$ for all $A\in\slin(V)$ is equivalent to the fact $\varphi^t\varphi\in\F I_6$. This is a necessary not sufficient condition for the automorphisms $\widetilde{\varphi}$ and $\theta$ to commute. They commute when $ \widetilde{\varphi}^{-1}\theta^{-1}\widetilde{\varphi}\theta\vert_{L_1+L_{-1}}=\id$; since $L_1+L_{-1}$ generate $L$
as an algebra.
For instance, if $\varphi$ is given by the diagonal matrix $\diag(\alpha_1,\alpha_2,\alpha_3,\alpha_4,\alpha_5,\alpha_6)$,
then $\widetilde{\varphi}$ commute with $\theta$ if and only if $ \alpha_{\sigma(1 )}\alpha_{\sigma( 2)}\alpha_{\sigma( 3)}=\alpha_{\sigma(4 )}\alpha_{\sigma(5 )}\alpha_{\sigma( 6)}  $ for all permutation $\sigma\in S_6$, which is equivalent to %$\alpha_i/\alpha_j\in\{\pm1\}$ for all $i,j$ and
the fact that $\frac1{\alpha_1}\varphi$ has order two and determinant equal to $1$.

In particular, if we denote $f_{ij}=\diag(\alpha_1,\alpha_2,\alpha_3,\alpha_4,\alpha_5,\alpha_6)$ for $\alpha_i=\alpha_j=-1$ and the remaining elements in the diagonal equal to $1$, the above implies that all of their extensions commute with both $\theta$ and $T_1$, so that we can consider the quasitorus $P_5:=T_1\cdot \langle \theta, \widetilde{f}_{12}, \widetilde{f}_{13}, \widetilde{f}_{14}, \widetilde{f}_{15}\rangle \cong \F^\times\times\mathbb{Z}_2^5$. In order to check that it is a MAD-group of $\aut(L)$ (and hence   conjugated to $\mathcal{Q}_{12}$ in \cite{e6}), it is enough to take into account that
\begin{itemize}
\item[(i)]
the centralizer of the torus $T_1$ is $\widetilde\GL(V)\cup \widetilde\GL(V)\theta$;
\item[(ii)] If $\varphi,\psi\in\GL(V)$, their extensions commute if and only if $ {\varphi}^{-1}{ \psi}^{-1} {\varphi} { \psi}
    %\in\{I_6,\omega I_6,\omega^2I_6\}$ for $\omega^3=1$,
    =\omega^i I_6$ for $\omega^3=1$ and some $i=0,1,2$;
    but in case $\varphi^2=I_6$, the    extensions  $\widetilde{\varphi}$ and $\widetilde{\psi}$ commute if and only if $\varphi$ and $\psi$ do;
    \item[(iii)] The element $\widetilde{f}_{12}\widetilde{f}_{34}\widetilde{f}_{56}=-\widetilde{I}_6=f_{-1}\in T_1$;
\item[(iv)] A matrix $A\in\Mat_{6\times6}(\F)$ commutes with all the order two diagonal matrices if and only if $A$ itself is a diagonal matrix.
\end{itemize}
The induced fine grading will be denoted by $\Gamma_5$.

%%%%%%%%%%%%%%%%%%%%%%%%%%%%%%%%%%%%%%%%%%%%%%%%%%%%%%%%%%%%%%%%%%%%%%%%%%%%%%%%%%%%%%%%%%%%%%%%%%%%

\subsection{Outer $ \Z_2^3\times\Z^2 $-grading}\label{subsec_gradZoutercuadZ2cubo}

According to \cite[Section~5.3]{e6}, the following is a  $\Z^2 \times \Z_2^3$-fine grading on $\mathbf{L}$, the algebra of type $E_6$ described in Subsection~\ref{subsec_modeloc4}.
Denote
$\sigma_1= \left(\begin{array}{cc}
0&1\\
1&0
\end{array}\right)
$
and
$\sigma_2= \left(\begin{array}{cc}
1&0\\
0&-1
\end{array}\right).
$
Then the automorphisms
$$
\begin{array}{l}
\tau_{\alpha,\beta}=
\diag({\alpha,\alpha,\frac1\alpha,\frac1\alpha,\beta,\beta,\frac1\beta,\frac1\beta}) ,\\
g_1=\diag({\sigma_1,\sigma_1,\sigma_1,\sigma_1}) ,\\
g_2=\diag({\sigma_2,\sigma_2,\sigma_2,\sigma_2}) ,
\end{array}
$$
constitute a MAD-group of $\Sp_8(\F)\cong\aut(\mathbf{L}_0)$ and hence
$
P_6:=  \langle \Theta,  {g_1}^\diamondsuit,  {g_2}^\diamondsuit \rangle \cdot \{ \tau_{\alpha,\beta}^\diamondsuit\mid \alpha,\beta\in\F^\times\} \cong\mathbb{Z}_2^3 \times  (\F^\times)^2$  is a MAD-group of $\aut(\mathbf{L})$.

Although in $\Sp_8(\F)$ it is very easy to find elements whose extensions move the elements in the MAD-group $P_6$, the disadvantage is that all of them commute with $\Theta$. That forces us to look for a different realization through the 5-grading model  of Subsection~\ref{subsec_modelo5grad}. Our quest is inspired in the following  key fact about the above grading: the fixed part of the two-dimensional torus must be a subalgebra of dimension $18$ (and rank $6$), and $\theta$ acts on it fixing a subalgebra of dimension $8$ (and rank $4$).

 Consider the elements of $\text{GL}_6(\F)$ given by
$$
\begin{array}{l}
%\psi_a= \text{diag}(\left(\begin{array}{cc} a&a' \\ -a'&a \end{array}\right), I_4), \text{ where $a,a'\in \F^\times$ satisfy $a^2+a'^2=1$}.\\
\psi_{a,a'}=  \left(\begin{array}{c|c}\begin{array}{cc} a&a' \\ -a'&a \end{array} & 0\\ \hline0&I_4\end{array}\right),\\
g'_1=\text{diag}(1,1,1,-1,1,-1),  \\
 g'_2=\text{diag}(1,1,1,-1,-1,1),
\end{array}
$$
where $a,a'\in \F$ satisfy $a^2+a'^2=1$. Then, we have a one-dimensional torus $T'_1=\{\widetilde{\psi}_{a,a'} \mid a,a'\in \F, a^2+a'^2=1 \}\le\aut(L)$, since this is a connected diagonalizable group (the ideal of $\F[x,y]$ generated by the polynomial $x^2+y^2-1$ is prime)
and a maximal quasitorus of $\aut(L)$ given by   %\margen{lema????}
$$
P'_6:=  \langle \theta, \widetilde{g}'_1, \widetilde{g}'_2 \rangle \cdot T_1\cdot T'_1\cong\mathbb{Z}_2^3 \times  (\F^\times)^2 .
$$
 For proving its maximality, take $f$ in the centralizer. We can replace $f$ with $f\theta$, if necessary, to have that there exists $\varphi\in\GL_6(\F)$ such that $f=\widetilde{\varphi}$. As $f$ commutes with $\theta$, then $\varphi\varphi^t=I_6$, the identity matrix. Also $f$ commutes with $\psi_{0,1}$, so that there are $x,y\in\F$
 such that $\varphi=\left(\begin{array}{c|c}\begin{array}{cc} x&y \\ -y&x \end{array} & 0\\ \hline0&B\end{array}\right)$. The fact $\varphi\varphi^t=I_6$ means that $x^2+y^2=1$, so that by replacing $\varphi$
 with   $\varphi\psi_{x,y}$  we can assume that $x=1$ and $y=0$.
 By item ii) above, $B$ commutes with $\diag(1,-1,-1,1)$ and with $\diag(1,-1,1,-1)$, so that there are $b_i$'s for $i=1,\dots,4$, such that $B=\diag(b_1,b_2,b_3,b_4)$. As $BB^t=I_4$, then $b_i^2=1$, but again using that $\varphi$ commutes with $\theta$ we get that $\det(B)=1$ and the number of $1$'s in its diagonal is either 0, or 2, or 4. Hence $B\in\span{\diag(1,-1,1,-1),\diag(1,-1,-1,1),-I_4}$ and $\widetilde\varphi\in P_6$ (note that $\psi_{-1,0}f_{-1}=
 \left(\begin{array}{c|c}I_2 & 0\\ \hline0&-I_4\end{array}\right)$).

%%%%%%%%%%%%%%%%%%%%%%%%%%%%%%%%%%%%%%%%%%%%%%%%%%%%%%
%%%%%%%%%%%%%%%%%%%%%%%%%    SECTION 5  %%%%%%%%%%%%%%
%%%%%%%%%%%%%%%%%%%%%%%%%%%%%%%%%%%%%%%%%%%%%%%%%%%%%%

\section{Calculating the Weyl groups}\label{sec_gruposWeyl}

In this section we compute $\W(\Gamma_i)=\aut(\Gamma_i)/\stab(\Gamma_i)$ for $i=1,\dots,6$.
Recall that if  $P_i$ is the MAD-group producing the grading $\Gamma_i$, then $G_i=\X(P_i)$ is the universal grading group,
and for each $f\in \aut(\Gamma_i)$, there is  a group isomorphism $\alpha_f\in\aut(G_i)$   such that $f(L_s)= L_{\alpha_f(s)}$ for all $s\in\supp(\Gamma_i)$. That allows us to identify $\W(\Gamma_i)$ with a subgroup of $\aut(G_i)$ and consequently to take a basis of the (always finitely generated) group and work with the matrices relative to that basis.
Concretely, if $G=\Z_p^m\times\Z^n$, we fix
\begin{equation}\label{eq_basedelgrupo}
\begin{array}{c}
\{(\bar1,\dots,\bar0, 0,\dots,0),\dots,(\bar0,\dots,\bar1, 0,\dots,0),\\
(\bar0,\dots,\bar0, 1,\dots,0),\dots,(\bar0,\dots,\bar0, 0,\dots,1)\}
 \end{array}
 \end{equation}
 as canonical basis.
We also fix scalars $\ii,\omega,\xi,\zeta\in\F$ such that $\ii^4=\zeta^{12}=\omega^3=\xi^9=1$, primitive roots of the unit.

%%%%%%%%%%%%%%%%%%%%%%%%%%%%%%%%%%%%%%%%%%%%%%%%%%%%%%%%%%%%%%%%%%%%%%%%%%%%%%%%%%%%%%%%%%%%%%%%%%%%%%%%%%%%%%%%%%%%%%%%

\subsection{Weyl group of the inner $\Z_2^3 \times \Z^2$-grading}\label{subsec_WeylZ2cuboZcuad}

Take $\Gamma_1$   the   $G_1=\mathbb{Z}_2^3\times\mathbb{Z}^2$-grading on $\e_6$ described in Subsection~\ref{subsec_gradZcuadZ2cubo}, which is produced by a quasitorus of inner automorphisms. Identify any element of $\aut(G_1)$ with its  matrix (by columns)
relative to the canonical basis (\ref{eq_basedelgrupo}),
 which is a $3+2$-block matrix
  whose  first three rows have coefficients in $\mathbb{Z}_2$ and the other ones in $\Z$.
\begin{prop}
The Weyl group   $\W(\Gamma_1)$ coincides with the set of block matrices
$$
\left\{\left(\begin{array}{c|c}A & C\\ \hline0&B\end{array}\right)\mid A\in\GL_3(\Z_2),B\in\langle\tau_1,\tau_2\rangle,
C\in\Mat_{3\times2}(\mathbb{Z}_2)\right\}
$$
where
$$
\tau_1=\left(\begin{array}{cc}0&-1\\1&-1\end{array}\right), \quad  \tau_2=\left(\begin{array}{cc}-1&1\\0&1\end{array}\right),
 $$
generate the dihedral group $D_3$. Thus
$$
\W(\Gamma_1)
\cong \Mat_{3\times2}(\Z_2) \rtimes (\GL_3(\mathbb{Z}_2) \times D_3).
$$
\end{prop}

\begin{proof}
First, note that the matrix of $\alpha_f$ has necessarily a zero block in the left lower corner, because $\alpha_f$
must apply finite order elements of $G_1$ into finite order elements.

Recall from \cite[Theorem~3.5]{EK11} that the Weyl group of the $\Z_2^3$-grading on $\C$  fills   $\aut(\Z_2^3)$, since the three elements $w_i$'s as in Equation~(\ref{eq_gradZ2cubodeC}) play exactly the same role. Thus this Weyl group can be identified with $\GL_3(\Z_2)$.  Since the automorphisms of $\C$ are extended to $\T(\C,\J)\cong\e_6$ using the Tits' construction and commuting with $T_2$, this implies that any $\left( \begin{array}{cc}
A & 0\\
0 & I_2
\end{array} \right)$ belongs to $\W(\Gamma_1)$.

Now note that the $\Z^2$-grading on $\J=M_3(\F)^+$ induces the $\Z^2$-grading on $\der(\J)\cong\slin_3(\F)$ which is the root decomposition of an algebra of type $A_2$. It is well known that the Weyl group of this grading is
isomorphic to  $D_3$. With our identifications,
if $\sigma\in S_3$, the map $E_{ij}\mapsto E_{\sigma(i)\sigma(j)}$ is an automorphism of the Jordan algebra $\J$. All these automorphisms must fill the Weyl group ($S_3$ is isomorphic to $D_3$), being $\tau_1$ the matrix corresponding to the permutation $(1,2,3)$ (since $E_{12}$ is applied into $E_{23}\in\J_{0,1}$ and    $E_{23}$   into $E_{31}\in\J_{-1,-1}$)
and $\tau_2$ the matrix corresponding to the permutation $(1,2)\in S_3$. Again these automorphisms can be extended to $\e_6$ commuting with $\{f_1,f_2,f_3\}$, appearing the elements $\left( \begin{array}{cc}
I_3 & 0\\
0 & B
\end{array} \right)$ in the Weyl group (and only those elements with such particular shape).

In order to find the remaining elements in the Weyl group, we need to change the model used for describing the grading. Then we
  consider the construction
  $\mathfrak{L} =\mathfrak{sl}(V)\oplus\mathfrak{sl}(W)\oplus (V\otimes\bigwedge^3W)$ given
  in Subsection~\ref{subsec_modeloa5masa1}.
  %, and fix $\xi,\ii\in \F$ such that $\xi^6=-1$ and $\ii^2=-1$.
  According to \cite[Section~3.5]{e6}, $\Gamma_1$  is equivalent to $\Gamma'_1$, the $\Z_2^3 \times \Z^2$-grading  on $\mathfrak{L}$ with grading automorphisms:
$$
\begin{array}{l}
 f_1'=\ii\left(\begin{array}{cc} 1&0\\0&-1 \end{array}\right) \times  \zeta\left(\begin{array}{cc}   I_3 & 0 \\ 0 & -  I_3 \end{array}\right), \\
 \text{$f_2'$ given by $\id$ on $\mathfrak{sl}(V)\oplus\mathfrak{sl}(W)$ and $-\id$ on $V\otimes\bigwedge^3W$,} \\
 f_3'=\ii\left(\begin{array}{cc} 0&1\\1&0 \end{array}\right) \times \zeta\left(\begin{array}{cc} 0 &   I_3 \\   I_3 & 0 \end{array}\right), \\
 s_{\alpha,\beta}'=I_2 \times \text{diag}(\alpha,\beta,\alpha^{-1}\beta^{-1},\alpha,\beta,\alpha^{-1}\beta^{-1}), \quad \alpha,\beta\in \F^\times,
\end{array}
$$
\noindent
if we recall that we are considering $\Sp(V)\times \SL(W)\leq\text{Aut}( \mathfrak{L})$.
Note that not only the quasitorus $P'_1=\{s'_{\alpha, \beta}\mid \alpha,\beta\in \F^\times\}\langle  f_1', f_2', f_3' \rangle$ can be considered  conjugated to $P_1$ through the identification, but   it is possible to take an isomorphism $\psi\colon \mathfrak{L} \to\e_6$ in a way such that $\psi^{-1}f_i\psi=f_i'$ for all $i=1,2,3$, due to the fact that a nontoral $\Z_2^3$-subgroup of $\aut(\e_6)$ is unique up to conjugation and that we have already proved that we can change any nonidentity element in that subgroup $\Z_2^3$ by any other one.
These arguments allow to work with the  other model.

For each $\sigma\in S_6$, denote by $p_\sigma\in\GL_6(\F)$ the matrix with $(i,\sigma(i))$  entry equal to $1$   for all $i=1,\dots,6$ and $0$ otherwise. Consider the automorphism given by
$\varphi=I_2\times p_\sigma\in \Sp(V)\times \SL(W)\leq\aut(\mathfrak{L})$ for the permutation $\sigma=(1,4)(3,6)\in S_6$.
It is a trivial verification that
 $$
 \begin{array}{ll}
 \varphi f'_1\varphi^{-1}=f'_1s'_{-1,1},\qquad&
 \varphi f_2'\varphi^{-1}=f_2',\\
 \varphi f_3'\varphi^{-1}=f_3',&
 \varphi s'_{\alpha,\beta}\varphi^{-1}=s'_{\alpha,\beta},
 \end{array}
 $$
which implies that $\varphi$ belongs to $\norm(P'_1)$ and that the induced element by its projection on the Weyl group of the grading is
$$\small{
 \left( \begin{array}{ccccc}
\bar 1 & 0 & 0 & \bar 1 & 0\\
0 & \bar 1 & 0 & 0 & 0\\
0 & 0 & \bar 1 & 0 & 0\\
0 & 0 & 0 & 1 & 0\\
0 & 0 & 0 & 0 & 1
\end{array} \right)}\in\W(\Gamma'_1).
$$
In a similar way, the automorphism $I_2\times p_\sigma$ for $\sigma=(2,5)(3,6)$ applies $f'_1$ into $f'_1s'_{1,-1}$
and does not move $f_2'$, $f_3'$ and   $s'_{\alpha,\beta}$.
As the automorphisms $f'_i$'s are interchangeable,    we find any possible
$\left( \begin{array}{cc}
I_3 & C\\
0 & I_2
\end{array} \right)$ belonging to $\W(\Gamma'_1)$.

It is important to observe that the two-dimensional torus $\psi^{-1}T_2\psi$   commutes with the automorphisms $\psi^{-1}f_i\psi=f_i'$ for all $i$, hence it is equal to $\{s'_{\alpha,\beta}\mid \alpha,\beta\in\F^\times\}$ (the only one satisfying that property), although possibly $\psi^{-1}s_{\alpha,\beta}\psi\ne s'_{\alpha,\beta}$.
Anyway, the map $\psi$ can be chosen verifying besides $\psi^{-1}s_{\alpha,\beta}\psi= s'_{\alpha,\beta}$ for all $\alpha^2=\beta^2=1$, if we compose $\psi$ with an automorphism  of $\mathfrak{L}$ with suitable projection $\left( \begin{array}{cc}
I_3 & C\\
0 & I_2
\end{array} \right)$.   This justifies that the previously found elements in  $\W(\Gamma'_1)$ also belong to $\W(\Gamma_1)$, and so any $\left( \begin{array}{cc}
A & C\\
0 & B
\end{array} \right)$ does.

\end{proof}

%%%%%%%%%%%%%%%%%%%%%%%%%%%%%%%%%%%%%%%%%%%%%%%%%%%%%%%%%%%%%%%%%%%%%%%%%%%%%%%%%%%%%%%%%%%%%%%%%%%%%%%%%%%%%%%%%%%%%%%%

\subsection{Weyl group of the $\Z_2^4 \times \Z$-grading}\label{subsec_WeylZ2cuartaZ}

Let $\Gamma_2$  be the $G_2$-grading on $\g(S_8,S_2)\cong\e_6$ described in Subsection~\ref{subsec_gradZZ2cuarta}, for
$G_2= \Z_2\times\Z_2^3 \times \Z$, the first factor $\Z_2$ corresponding to the outer automorphism $\varrho$ coming from $S_2$,  the $\Z_2^3$-grading induced by the three order two automorphisms  $\{F_1,F_2,F_3\}\subset\aut(\e_6)$ extending those ones of the paraoctonion algebra $S_8$, and the $\Z$-grading as in Equation~(\ref{eq_laZgrad}) produced by the torus $\{t_\alpha\mid \alpha\in\F^\times\}$.
Identify any element of $\aut(G_2)$ with its $1+3+1$-block   matrix relative to the basis (\ref{eq_basedelgrupo}).

\begin{prop}
The Weyl group   $\W(\Gamma_2)$ coincides with
$$
\left\{\left(\begin{array}{c|c|c} \bar1&0&b \\ \hline 0&A&D\\ \hline 0&0&c \end{array}\right) \mid b\in\{\bar0,\bar1\},c\in\{\pm1\}, A\in\GL_3(\Z_2),D\in\Mat_{3\times1}(\Z_2)\right\}.
$$
Thus
$$ \W(\Gamma_2)
\cong  \Z_2^4 \rtimes (\GL_3(\Z_2)\times\Z_2).$$

\end{prop}

\begin{proof}
First note that the restriction of our grading to the subalgebra $\fix\varrho\cong\f_4$ gives a   $\Z_2^3 \times \Z$-fine grading
on $\f_4$. As the algebraic groups $\aut(\alb)$ and $\aut(\f_4)$ are isomorphic, we can apply the results in \cite[Theorem~4.6]{EK11} about the Weyl group of the fine $\Z_2^3 \times \Z$-grading on $\alb$ to guarantee that the Weyl group of the corresponding grading on $\f_4$ is the whole group $\aut( \Z_2^3 \times \Z)$. This implies that
 all the elements of the form
$\left(\begin{array}{c|c|c} \bar1&0&0 \\ \hline 0&A&D\\ \hline 0&0&c \end{array}\right)$ belong to $\W(\Gamma_2)$. 

Notice that the torsion subgroup $\Z_2^4$ of $G_2$ is preserved by the action of $\W(\Gamma_2)$, so the zero blocks of the third row must appear.
Observe also that if $(a_{ij})$ is the matrix of $\alpha_f\in\aut(G_2)$ related to $f\in\aut(\Gamma_2)$, not only this automorphism $f$ belongs to the normalizer of the MAD-group producing the grading, but we can specify its action:
$$
\begin{array}{l}
f^{-1}\varrho f=\varrho^{a_{11}}F_1^{a_{12}}F_2^{a_{13}}F_3^{a_{14}}t_{-1}^{a_{15}},\\
 f^{-1}F_{i-1} f=\varrho^{a_{i1}}F_1^{a_{i2}}F_2^{a_{i3}}F_3^{a_{i4}}t_{-1}^{a_{i5}},\qquad \forall i=2,3,4,\\
 f^{-1}t_\alpha f=t_{\alpha^c},\qquad c=a_{55}\in\{\pm1\},\forall\alpha\in\F^\times.
 \end{array}
 $$
 In particular, as $f$ applies inner automorphisms into inner automorphisms, then the zero blocks of the first column appear. Furthermore, in order to prove that the block located in $(1,2)$-position is zero, it is enough to check that $\varrho$ is not conjugated to $\varrho F$ for any $F\in\span {F_1,F_2,F_3}$. Indeed, if we think of $\e_6$ as $\fix\varrho\cong\der(\alb)$ direct sum $\text{antifix}\,\varrho\cong\alb_0$, note that   such an $F$ comes from an order two automorphism of the octonion algebra (fixing a quaternion algebra $Q$), first extended to the Albert algebra (fixing $\mathcal{H}_3(Q )$, of dimension $15$) and then to $\f_4$ (fixing  $C_3$, seen as $\der(\mathcal{H}_3(Q ))$, direct sum with an ideal of type $A_1$, %\margen{explico m\'{a}s?}
 so of dimension $24$). Thus $F$ is of type $2A$, and, what is more useful for us: the subalgebra fixed by $\varrho F$ has dimension $24+12$ (the fixed part of $\der(\alb)$ and the antifixed part of $\alb_0$), so that it is of type $C_4$ instead of $F_4$, as we were proving.

For ending the proof, we need some other automorphisms. We consider again the model $\g(S_8,S_2)\cong\e_6$ and take  the following maps,
\begin{equation}\label{eq_losautomorfismosdeDiego}
\begin{array}{rccl}
\Psi_i\colon  &\e_6 & \longrightarrow & \e_6 \\
&\iota_i(x\otimes y) & \longmapsto & \iota_i(x\otimes y)\\
&\iota_{i+1}(x\otimes y) & \longmapsto & \ii \iota_{i+1}(x\otimes \eta(y))\\
&\iota_{i+2}(x\otimes y) & \longmapsto & -\ii \iota_{i+2}(x\otimes \eta(y))\\
&t & \longmapsto & t
\end{array}
\end{equation}
for every $t\in \mathfrak{tri}(S_8) \oplus \mathfrak{tri}(S_2)$ and for all $x\in S_8$, $y\in S_2$, $i=0,1,2$, where $\eta\colon S\rightarrow S$ is the linear map defined by $\eta(e_1)=e_1$ and $\eta(e_2)=-e_2$. It is a straightforward computation that the maps $\Psi_i$'s are indeed order $4$ automorphisms of $\e_6$. Observe that
  $\Psi_0$ belongs to $\aut(\Gamma_2)$, since it preserves the $\Z$-grading, it commutes with all the $F_i$'s, and it verifies $\Psi_0^{-1}\varrho\Psi_0=\varrho t_{-1}$ (so it interchanges $L_{(\bar0,*)}$ with $L_{(\bar1,*)}$). Then  its projection on $\W(\Gamma_2)$ is
$$
\small{\left( \begin{array}{ccc}
\bar1 & 0 & \bar1\\
0 & I_3 & 0\\
0 & 0 &  1
\end{array} \right), }
$$
what finishes the proof.

\end{proof}

%%%%%%%%%%%%%%%%%%%%%%%%%%%%%%%%%%%%%%%%%%%%%%%%%%%%%%%%%%%%%%%%%%%%%%%%%%%%%%%%%%%%%%%%%%%%%%%%%%%%%%%%%%%%%%%%%%%%%%%%

\subsection{Weyl group of the $\Z_2 \times \Z^4$-grading}\label{subsec_WeylZ2Zcuarta}

Let $\Gamma_3$  be the $G_3=\Z_2 \times \Z^4$-grading on $\g(S_8,S_2)$ described in Subsection~\ref{subsec_gradZcuartaZ2}.
The MAD-group producing the grading is $\Diag(\Gamma_3)=P_3=\span{\{\varrho\}\cup T_4}$.
Identify any element of $\aut(G_3)$ with its  $1+4$-block matrix relative to the canonical basis (\ref{eq_basedelgrupo}).

\begin{prop}
The Weyl group   $\W(\Gamma_3)$ coincides with
\begin{equation}\label{eq_weyldegamma3}
 \W(\Gamma_3) = \left\{ \left(\begin{array}{c|c} \bar1& \begin{array}{cccc} a&b&a&b \end{array} \\ \hline 0&A \end{array}\right) \mid   A\in\W_{\f_4},\,   a,b\in\Z_2  \right\},
 \end{equation}
where $\W_{\f_4}$ is the Weyl group of the Cartan grading of $\f_4$, generated by
$$\small{
s_1=\left(
\begin{array}{cccc}
 0 & -1 & 1 & -1 \\
 1 & -1 & 1 & 0 \\
 0 & 0 & 1 & 0 \\
 0 & 0 & 0 & 1
\end{array}
\right),\,s_2=\left(
\begin{array}{cccc}
 -1 & 1 & 0 & -1 \\
 0 & 1 & 0 & 0 \\
 0 & 0 & 1 & 0 \\
 0 & 0 & 0 & 1
\end{array}
\right),}
$$
$$\small{
s_3=\left(
\begin{array}{cccc}
 0 & 0 & 1 & -1 \\
 0 & 1 & 0 & 0 \\
 1 & 0 & 0 & 1 \\
 0 & 0 & 0 & 1
\end{array}
\right),\,s_4=\left(
\begin{array}{cccc}
 1 & 0 & -1 & 1 \\
 0 & 1 & -1 & 0 \\
 0 & 0 & 0 & -1 \\
 0 & 0 & 1 & -1
\end{array}
\right).}
$$
Thus
$$
\W(\Gamma_3)
\cong  \Z_2^2\rtimes  \W_{\f_4} \cong\Z_2^2\rtimes( (\Z_2^3\rtimes S_4)\rtimes S_3 ),
$$
where $S_n$ denotes the symmetric group of $n$ elements.
 \end{prop}

\begin{proof}
If $f\in\aut(\Gamma_3)$ and $\alpha_f$ denotes its related matrix, it is clear that the first column must be as claimed, since the torsion subgroup $\Z_2$ of $G_3$ is invariant by automorphisms.

 Recall that the $\Z_2$-grading produced by $\varrho$ allows to identify $\g(S_8,S_2)$ with $\e_6=\der(\alb)\oplus\alb_0$  in such a way that
 the restriction of $T_4$ is the Cartan grading of $\der(\alb)=\f_4$ and
 any    automorphism of $\der(\alb)$ can be  extended to $\e_6$ commuting with $\varrho$.
 In particular we extend the automorphisms of $\f_4$ which normalize the maximal torus, so that
   the elements
 $\left(\begin{array}{c|c} \bar1& 0 \\ \hline 0&A \end{array}\right)$ belong to $\W(\Gamma_3)$ for all $A\in\W_{\f_4}$. Conversely, if an element in  $\W(\Gamma_3)$  has such a particular shape, then the block $A$ necessarily belongs to $\W_{\f_4}$, taking into account that commuting with $\varrho$ forces any automorphism of $\e_6$ to be an extension of an automorphism of $\f_4$. The generators $\{s_i\mid i=1,\dots,4\}$ of $\W_{\f_4}$ that we have chosen are,
 respectively, the ones related to the automorphisms $\psi_{(1,2,3)}$,  $\psi_{(2,3)}$, $\psi_c$ and the extension of $\tau$ described in \cite[Theorem~4.2]{EK11} and with the notations used there.

The  automorphisms $\Psi_i$ considered in Equation~(\ref{eq_losautomorfismosdeDiego}) belong to $\aut(\Gamma_3)$ for all $i=0,1,2$, being their projections in $\W(\Gamma_3)$:
$$\tiny{
\alpha_{\Psi_0}=\left( \begin{array}{c|c}
\bar 1 & \begin{array}{cccc}\bar0 & \bar1 & \bar0 & \bar1\end{array}\\
%\bar 1 & \begin{array}{c }\bar0 \, \bar1 \, \bar0 \, \bar1\end{array}\\
\hline
0&I_4
\end{array} \right),
%\;
\alpha_{\Psi_1}=\left( \begin{array}{c|c}
\bar 1 & \begin{array}{cccc}\bar1 & \bar0 & \bar1 & \bar0\end{array}\\
\hline
0&I_4
\end{array} \right),
%\;
\alpha_{\Psi_2}=\left( \begin{array}{c|c}
\bar 1 & \begin{array}{cccc}\bar1 & \bar1 & \bar1 & \bar1\end{array}\\
\hline
0&I_4
\end{array} \right).}
$$
Therefore, all the elements of the right side in Equation~(\ref{eq_weyldegamma3}) belong to the Weyl group.

To prove the equality now is equivalent to prove that there is not $f\in\norm(P_3)$ such that
$f^{-1}\varrho f=\varrho t$ for $t=t_{(-1)^{b_1},(-1)^{b_2},(-1)^{b_3},(-1)^{b_4}}\in T_4$  with $b_i\in \{0,1\}$
when $b_1\neq b_3$ or $b_2\neq b_4$. Recall that $t$ acts in $L_{(\bar n_0,n_1,n_2,n_3,n_4)}$ with eigenvalue $(-1)^{b_1n_1+b_2n_2+b_3n_3+b_4n_4}$.
But observe that $\varrho$ is not even conjugated to $\varrho t$, since the latter automorphism has type $2D$.
Indeed, the subalgebra fixed by $\varrho t $ is
$$
\begin{array}{rl}
\fix(\varrho t)=&\fix t\vert_{\mathfrak{tri} (S_8)}\oplus (\bigoplus_{i=0}^{2}\iota_i(S_8\otimes 1))\cap \text{fix}(t)\\
& \oplus (\bigoplus_{i=0}^{2}\iota_i (S_8 \otimes (e_1-e_2)))\cap \text{antifix}(t).
\end{array}
$$
If $x\in B$, being $B$ the canonical basis of $S_8$, $t$ acts in $\iota_i (x\otimes s)$ with eigenvalue  either $1$ or $-1$ independently of the considered element $s\in S_2$, so that just one of the two elements $\iota_i (x\otimes 1)$
and $\iota_i (x\otimes (e_1-e_2))$ is fixed by $\varrho t$.
Hence $\dim \fix(\varrho t)= \dim \fix(t |_{ \mathfrak{tri} (S_8)})+24$,
 which will be different from $52$ whenever $t|_{\mathfrak{tri} (S_8) }\ne\id$.
But if $b_1\neq b_3$, $t$ acts on $t_{v_2,v_3}\in L_{(\bar 0,1,0,1,0)}$ with eigenvalue $-1$, and if $b_2\neq b_4$, $t$ acts on $t_{e_2,u_2}\in L_{(\bar  0,0,1,0,1)}$ with eigenvalue $-1$, proving our assertion.
\end{proof} 

%%%%%%%%%%%%%%%%%%%%%%%%%%%%%%%%%%%%%%%%%%%%%%%%%%%%%%%%%%%%%%%%%%%%%%%%%%%%%%%%%%%%%%%%%%%%%%%%%%%%%%%%%%%%%%%%%%%%%%%%%%

\subsection{Weyl group of the $\Z_3^2 \times \Z^2$-grading}\label{subsec_WeylZ3cuadZcuad}

In   Subsection~\ref{subsec_gradZcuadZ3cuad} we described the grading  $\Gamma_4$ as the one produced by the MAD-group $P_4=\langle h_1, h_2, \T_2  \rangle\le\aut(\T(\C,\J))$.
We work with $2+2$-block matrices relative to the canonical basis of $G_4=\Z_3^2 \times \Z^2$.

\begin{prop}
The Weyl group   $\W(\Gamma_4)$ coincides with
\begin{equation}\label{eq_elgrupodeweyldelostreses}
\left\{\left(\begin{array}{c|c} A& \begin{array}{cc} a&a\\b&b \end{array} \\ \hline 0&B \end{array}\right)\mid A\in\GL_2(\Z_3),B\in\langle \sigma,\tau\rangle,
a,b\in\mathbb{Z}_3\right\}
\end{equation}
where
$$
\sigma=\left(\begin{array}{cc}1&-1\\1&0\end{array}\right) \quad \text{and} \quad \tau=\left(\begin{array}{cc}1&-1\\0&-1\end{array}\right)
 $$
 generate the dihedral group $D_6$.
Thus
$$ \W(\Gamma_4)
\cong  \Z_3^2 \rtimes (\GL_2(\mathbb{Z}_3) \times D_6).
$$
\end{prop}

\begin{proof}
First of all, every automorphism of $G_4$ preserves the torsion subgroup, which forces to have a zero block in the $(2,1)$-position.

It has been computed in \cite{normPauli} that the Weyl group of the Pauli $\Z_n^2$-grading on $\slin_n(\F)$ is the group $\{A\in\Mat_{2\times2}(\Z_n)\mid \det(A)=\pm1 (\text{mod}\,n)\}$, isomorphic to two copies of $\SL_2(\Z_n)$, which in our case $n=3$ ($\slin_3(\F)\cong\der(\J)$)  coincides with the group $\GL_3(\Z_2)$.
Thus, the elements
$$\left( \begin{array}{cc}
A & 0\\
0 & I_2
\end{array} \right)$$
belong to $W(\Gamma_4)$ for all $A\in \SL_2(\Z_3)\cup \tiny{\left( \begin{array}{cc}
\bar1 & 0\\
0 & \bar2
\end{array} \right)} \SL_2(\Z_3)=\GL_3(\Z_2)$,
 since $\aut(\J)$ and $\aut(\der(\J))$ are isomorphic algebraic groups, so that we can extend any automorphism of $\J$ to an automorphism of $\T(\C,\J)\,(\cong\e_6)$ which commutes with $\T_2$.

Similarly, if an automorphism $\psi$ belonging to $\W(\Gamma_4)$ has matrix
$\left( \begin{array}{cc}
I_2 & 0\\
0 & B
\end{array} \right)$, then $\psi$ commutes with both $h_1$ and $h_2$, that is, $\psi$ belongs to the centralizer   $\Cent\span{h_1, h_2}\cong\langle h_1, h_2 \rangle \times \aut(\C)$ (\cite{e6}). Hence, certain $\psi h_1^{n_1}h_2^{n_2}\in\aut(\C)$ and, so, it is in the normalizer of $\T_2$. Thus, $B$ belongs to the Weyl group of the $\Z^2$-grading on $\C$, that is isomorphic to Weyl group of  the Cartan grading on $\g_2$, and consequently it is isomorphic to the dihedral group $D_6$.
To concrete the representation in terms of our fixed basis of $G_4$, recall from \cite{EK11}
that this Weyl group is generated by the classes of the following  automorphisms of $\C$:
$$
\begin{array}{ll}
 \rho\colon  \C \rightarrow \C,&\quad e_j \mapsto e_j, u_{i}\mapsto u_{i+1},v_{i} \mapsto v_{i+1},\\
 \psi_1\colon  \C \rightarrow \C,&\quad e_1 \leftrightarrow e_2,u_i \leftrightarrow v_i,\\
 \psi_2\colon  \C \rightarrow \C,&\quad e_j \mapsto e_j,u_1 \mapsto -u_1, u_2 \leftrightarrow u_3, v_1 \mapsto -v_1, v_2 \leftrightarrow v_3,
\end{array}
$$
where $i=1,2,3$ (mod 3) and $j=1,2$.
  Their extensions  to $\e_6$ belong to $\W(\Gamma_4)$ and have related matrices, respectively,
$$\small{\left( \begin{array}{ccc}
I_2 & 0 & 0\\
0 & -1 & 1\\
0 & -1 & 0
\end{array} \right), \quad
\left( \begin{array}{ccc}
I_2 & 0 & 0\\
0 & -1 & 0\\
0 & 0 & -1
\end{array} \right), \quad
\left( \begin{array}{ccc}
I_2 & 0 & 0\\
0 & 1 & -1\\
0 & 0 & -1
\end{array} \right)}.
$$

We need some knowledge about the various order $3$ elements in the torus of $P_4$. Let us denote by $t_{a,b}\in\T_2$ the automorphism of $\T(\C,\J) $ which extends the automorphism of $\C$
whose action on the canonical basis of $\C$ (see Subsection~\ref{subsec_involvedstructures}) is diagonal with scalars
$$
\{1,1,a,b,\frac1{ab},\frac1a,\frac1b,ab\}.
$$
Its extension to $\der(\C)$ is also diagonal in the basis of $\der(\C)$ given by
$
\{D_{u_1,v_1},D_{u_2,v_2},D_{e_1,u_1},$
$D_{u_2,e_1},D_{e_1,u_3},D_{e_1,v_1},D_{e_1,v_2},D_{e_1,v_3}
D_{u_1,v_2},D_{u_1,v_3},D_{u_2,v_1},D_{u_2,v_3},D_{u_3,v_1},D_{u_3,v_2}\},
$
with respective eigenvalues
$$
\{1,1,a,b,\frac1{ab},\frac1a,\frac1b,ab,\frac ab,a^2b,\frac ba,ab^2,\frac1{a^2b},\frac1{ab^2}\}.
$$
This list allows us to distinguish the conjugacy classes of the elements in the torus by the dimension of its fixed part. Thus, if we split the elements of order $3$ in $\T_2$   as $\T_2^1\cup\T_2^2$ for $\T_2^1=\{t_{\omega,\omega},t_{\omega^2,\omega^2}\}$  and $\T_2^2=\{t_{a,b}\mid a^3=b^3=1, a\ne b\}\setminus\{\id\}$, then the elements in $\T_2^1$ fix subalgebras of dimension $24$, so that they are order 3 automorphisms of type $3C$
 and 
the elements in $\T_2^2$ belong to the class $3B$. We leave it to the reader to verify that $h_it$ is of type $3D$ if $t\in \T_2^1$
and
of type $3C$ if $t\in \T_2^2$, for all $i=1,2$. This has an immediate consequence on the shape of the possible elements in the Weyl group. If
$$
\left(\begin{array}{c|c}I_2& \begin{array}{cc} a&b\\c&d \end{array} \\ \hline 0&I_2 \end{array}\right)\in \W(\Gamma_4),
$$
this  forces $h_1$ to be conjugated to $h_1t_{\omega^{2a},\omega^{2b}}$, and
  $h_2$ to be conjugated to $h_2t_{\omega^{2c},\omega^{2d}}$, so that, by the arguments above, necessarily $a=b$ and $c=d$ (modulo 3).

In order to finish the proof, we have to show that every element in the right side of   Equation~(\ref{eq_elgrupodeweyldelostreses})
belong to the Weyl group, but it is sufficient to find one automorphism
  $\psi\in\aut(\Gamma_4)$ with related matrix $\small{\left(\begin{array}{c|c}I_2& \begin{array}{cc} \bar1&\bar1\\\bar0&\bar0 \end{array} \\ \hline 0&B \end{array}\right)}$ for some $B\in\span{\sigma,\tau }$.

Within our search, we are going to change the viewpoint recurring to a different way of realizing the grading, as described in Subsection~\ref{subsec_modeloAdams}. The advantage is a greater disposability of automorphisms.
Take as
$\psi=\Psi    \left(\begin{array}{ccc}
0 & 1 & 0\\
0 & 0 & 1\\
1 & 0 & 0
\end{array}\right) \in \Psi(\GL_3(\F))\le\aut(\L)$.
%, if $\zeta^6=1$.
This automorphism $\psi$ commutes with $H_1$ and $H_2$ by construction, and verifies
$\psi T_{\alpha,\beta}\psi^{-1}=T_{\beta,\frac1{\alpha\beta}}$, in particular $\psi\in\norm(\mathcal{Q}_2)$ and then it can be projected on our Weyl group. But this cannot be done without making completely precise the identification:
In spite that $P_4$ and $\mathcal{Q}_2$ are conjugated, the conjugating automorphism does not necessarily apply $h_i$ into $H_i$ and $t_{a,b}$ into $T_{a,b}$, although it always applies $\T_2$ into $\{ T_{a,b}\mid a,b\in\F^\times\}$ (the only two-dimensional torus contained in $P_4$). First note that (recall $\xi^9=1$)  the subgroup $\span{H_1 T_{\xi, \xi},H_2}$($\cong\Z_3^2$) has every order three element of type $3D$, as proved in \cite[Lemma~11]{e6}. Besides the elements of type $3B$ in the accompanying torus are $T_{\omega,1}$ and its square. To avoid the inconvenience of the kernel of $\Psi$, consider the following isomorphism of two-dimensional algebraic torus
$$
\begin{array}{rcl}
(\F^\times)^2/\span{(\omega,\omega)}&\to&(\F^\times)^2\\
\overline{(\alpha,\beta)}&\mapsto& (\alpha\beta^2,\alpha/\beta).
\end{array}
$$
Thus consider $T_{\alpha,\beta}=:T'_{\alpha\beta^2,\alpha/\beta}$, $H_1':=H_1 T_{\xi, \xi}$ and $H_2':=H_2$.
With this notation the order three elements in the torus of type $3C$ are $\{T'_{\omega,\omega},T'_{\omega^2,\omega^2}\}$, and the other ones are of type $3B$. Now we will see that
there is an isomorphism $\Upsilon\colon \L\to\T(\C,\J)$ such that $P_4=\Upsilon \mathcal{Q}_2\Upsilon^{-1}$,
$h_i=\Upsilon H_i'\Upsilon^{-1}$ and $t_{\alpha,\beta}=\Upsilon T'_{\alpha,\beta}\Upsilon^{-1}$ if $\alpha^3=\beta^3=1$. As obviously the equations
  $$
  \begin{array}{l}
  \psi T'_{\alpha,\beta}\psi^{-1}=T'_{ \frac1{\alpha\beta},\alpha},\\
  \psi H_1'\psi^{-1}=H'_1T'_{\omega,\omega},\\
  \psi H_2'\psi^{-1}=H'_2,
  \end{array}
  $$
  hold, then $\Upsilon \psi\Upsilon^{-1}\in\norm(P_4)$ has as related matrix
  $$
  \left(\begin{array}{c|c}I_2& \begin{array}{cc} \bar1&\bar1\\\bar0&\bar0 \end{array} \\ \hline 0&  \begin{array}{cc} -1 & 1\\ -1 & 0\end{array}
  \end{array}\right),
  $$
as searched.

The existence of $\Upsilon$ is justified by the following argument: the elements of type $3D$ in $\mathcal{Q}_2$ are
just $\{{H'_1}^i{H'_2}^j{T'_{\omega,\omega}}^k\mid i,j,k\in\{0,1,2\},(i,j)\ne(0,0)\}$. An automorphism passing from $P_4$ to $\mathcal{Q}_2$ necessarily applies $h_1$ into one of these elements. By using previous elements in $\W(\Gamma_4)$ (any nonidentity element in $\span{H_1',H_2'}$ can be moved into $H_1'$ by an element in the normalizer),  we can replace the above automorphism with another one applying $h_1$ into $H'_1{T'_{\omega,\omega}}^j$ for some $j\in\{0,1,2\}$, and then by composing it with $\psi^{-j}$, with a third one which applies $h_1$ into $H'_1$.
\end{proof}

%%%%%%%%%%%%%%%%%%%%%%%%%%%%%%%%%%%%%%%%%%%%%%%%%%%%%%%%%%%%%%%%%%%%%%%%%%%%%%%%%%%%%%%%%%%%%%%%%%%%%%%%%%%%%%%%%%%%%%%%

\subsection{Weyl group of the $\mathbb{Z}_2^5\times\mathbb{Z}$-grading}\label{subsec_WeylZZ2quinta}

Take $\Gamma_5$ the $G_5=\mathbb{Z}_2^5\times\mathbb{Z}$-grading described in Subsection~\ref{subsec_gradZZ2quinta}, produced by $P_5=\span{\theta, \widetilde{f}_{12}, \widetilde{f}_{13}, \widetilde{f}_{14}, \widetilde{f}_{15}}\cdot T_1$. Take the canonical basis of $G_5$
in order to identify the elements of $\W(\Gamma_5)\le\aut(G_5)$ with their matrices relative to it, written by blocks $1+4+1$.
\begin{prop}\label{prop_el5}
The Weyl group $\W(\Gamma_5)$ coincides with the set
$$
\left\{\left(
\begin{array}{c|c|c}
\bar1& \begin{array}{cccc}a&b&c&d\end{array}  &e\\
%\bar1& a\,b\,c\,d  &e\\
\hline
0&A& \kappa(A) \\
\hline
0& 0 & \pm1
\end{array}\right) \mid A\in \Sp_4(\mathbb{Z}_2),\,a,b,c,d,e\in\mathbb{Z}_2 \right\} ,
$$
where $\Sp_4(\mathbb{Z}_2)=\{A\in\Mat_{4\times4}(\Z_2)\mid ACA^t=C\}$, for
$C=\small{\left(
\begin{array}{cccc}
 \bar0 & \bar1 & \bar1 & \bar1 \\
 \bar1 & \bar0 & \bar1 & \bar1 \\
 \bar1 & \bar1 & \bar0 & \bar1 \\
 \bar1 & \bar1 & \bar1 & \bar0
\end{array}
\right)}$,
$$
\kappa_0(a_1,a_2,a_3,a_4)=
\begin{cases}\bar0\, \text{ if   $\  \vert\{i\mid a_i=\bar1\}\vert=1,2$},\\
\bar1\, \text{ if    $\ \vert\{i\mid a_i=\bar1\}\vert=3,4$},
\end{cases}
$$
%%o sea (-1)^{E[(s-1)/2]} para E=parte entera y s=ese cardinal
and $\kappa(A)=\kappa((a_{ij}))=\left(\begin{array}{c}
\kappa_0(a_{11},a_{12},a_{13},a_{14})\\
\kappa_0(a_{21},a_{22},a_{23},a_{24})\\
\kappa_0(a_{31},a_{32},a_{33},a_{34})\\
\kappa_0(a_{41},a_{42},a_{43},a_{44})
 \end{array}\right)$ .
Therefore,
$$
 \W(\Gamma_5) \cong  (\Sp_4(\mathbb{Z}_2)\times\mathbb{Z}_2) \ltimes \mathbb{Z}_2^5.
$$
\end{prop}

\begin{proof}
Since the action by conjugation of $\aut(\Gamma_5)$ on $P_5$ takes inner  automorphisms to inner   automorphisms,  any element in $\W(\Gamma_5)$ must have zero blocks in the positions $(2,1)$ and $(3,1)$.
Also, the zero block in the position $(3,2)$ is a consequence of the fact that the torsion group of $G_5$ is preserved.

The centralizer of $\theta$ is known to be isomorphic to $\Sp_8(\F)\cdot\span{1,\theta}\cong\Sp_8(\F)\times \Z_2$, what implies that every automorphism commuting with $\theta$ is an extension of some automorphism of $\mathfrak{sp}\,(8)\equiv\mathfrak{c}_4$ and conversely. That means that any automorphism $f$ such that
$\alpha_f=\left(\begin{array}{c|c|c}
\bar1& 0  &0\\
\hline
0&A& B \\
\hline
0& 0 & c
\end{array}\right)\in\W(\Gamma_5)$
verifies that $f\vert_{\fix\theta}$ belongs to the automorphism group of the $\mathbb{Z}\times\mathbb{Z}_2^4$-grading on $\mathfrak{c}_4$ produced by the restriction to $\fix\theta$ of the automorphisms   in $P_5 $. Its Weyl group has been proved in \cite[Theorem~5.7]{EK12} to be $\Sp_4(\mathbb{Z}_2)\times\mathbb{Z}_2$.
For completeness, we wish to get  a matricial expression in terms of our basis.  It is possible to take $c=-1$, considering for that purpose $\phi\colon\bigwedge^6V\to\F$ a fixed nonzero linear map, and then the automorphism of $L$ extending $\rho\colon\bigwedge^3V\to\bigwedge^3V^*$ given by $\span{x,\rho(y)}=\phi(y\wedge x)$.
Now assume that $f\in\norm(P_5)$ commutes with $\theta$ and with $T_1$ (so $c=1$). Hence there is $X\in\GL_6(\F)$ such that either $f$ or $f\theta$ equals $\widetilde X$.
Any inner order 2 element in $P_5$ is  $\widetilde Y$ for some  $Y\in \mathcal{D}=\{\diag(\varepsilon_1,\dots,\varepsilon_6)\mid \varepsilon_i^2=1,\Pi_{i=1}^6\varepsilon_i=1\}$.
As $Xf_{ij}X^{-1}\in \mathcal{D}$ for all $i,j$
(see item (ii) in Subsection~\ref{subsec_gradZZ2quinta}), this implies that $X$ is a permutation matrix (in each row and in each column, there is   only one nonzero entry). Thus there is $\sigma\in S_6$ such that $X=\diag(a_1,\dots,a_6)p_\sigma$ (recall that   $p_\sigma=(\delta_{\sigma(i)j})_{ij}$ for $\delta$ the Kronecker delta) and the action on $\mathcal{D}$ is given by
$$
X\diag(\varepsilon_1,\dots,\varepsilon_6) X^{-1}=\diag(\varepsilon_{\sigma(1)},\dots,\varepsilon_{\sigma(6)}).
$$
Note that, for each $Y=\diag(\varepsilon_1,\dots,\varepsilon_6)\in\mathcal{D}\setminus\{I_6\}$, the automorphism $\widetilde Y$   is of type $2A$ if the cardinal   $\vert\{i\in\{1,\dots,6\}\mid \varepsilon_i=-1\}\vert$ is either $2$ or $6$, and of type
$ 2B$ otherwise. All the extensions of these elements are obviously conjugated by suitable elements in $\norm(P_5)$, as well as all the elements of type $2A$ different from $f_{-1}=-\widetilde\id$. In particular, there is not a   subgroup of type $\Z_2^4$ which is  invariant for $\norm(P_5)$. Anyway, $S_6$ does act in the quotient $\mathcal{D}/\span{-\id}\cong\Z_2^4$ (the action $\sigma\cdot  \overline{Y}=\overline{p_\sigma Yp_\sigma^{-1}}$ is well defined) and it can be easily checked that for any cycle    $\sigma=(i,j)\in S_6$, the matrix $X_\sigma$ relative to the basis
$\{\overline{f}_{12}, \overline{f}_{13}, \overline{f}_{14}, \overline{f}_{15}\}$ verifies $X_\sigma CX^t_\sigma=C$.
The key point now is that the column matrix $B$ in $\alpha_f$ is completely determined by $A=X_\sigma$.
If the $i$th row of $A$ ($i=1,2,3,4$) is given by $(a,\,b,\,c,\,d)\in\Z_2^4$, then 
$f_{1,i+1}$ is sent to
$\diag((-1)^{a+b+c+d},(-1)^{a},(-1)^{b},(-1)^{c},(-1)^{d},1)$, which has four $-1$'s if and only if $(a,\,b,\,c,\,d)$ has three  or four $\bar1$'s.

Finally, we wish to know if every order two outer automorphism in $P_5 $ is conjugated to $\theta$ by an element in the normalizer, so that $a,b,c,d,e$ could take any value in $\{\bar0,\bar1\}$. The nontrivial point is to choose such an element inside
 the normalizer of the MAD-group, because of course we know that all these order two outer automorphisms are conjugated to $\theta$. Indeed, otherwise there would exist a fine $\mathbb{Z}\times\mathbb{Z}_2^4$-grading on $\f_4$, what does not happen according to \cite{f4}.

  The required automorphisms are easy to find by working in our model.
Note, if
$\varphi=\text{diag}(\alpha_{ 1},\alpha_{2 },\alpha_{ 3},\alpha_{ 4},\alpha_{ 5},\alpha_{ 6})\in\GL_6(\F)$, that
$\widetilde\varphi$ commutes with every inner automorphism of $P_5 $, although it does not commute necessarily with $\theta$. More precisely, $\widetilde\varphi\theta\widetilde\varphi^{-1}\theta^{-1}$ acts in $L_0$ as $\Ad(\varphi\varphi^t)$ and it acts in $e_{\sigma(1)}\wedge e_{\sigma(2)}\wedge e_{\sigma(3)}\in L_1$ with eigenvalue
$\frac{\alpha_{ \sigma(1)}\alpha_{\sigma(2) }\alpha_{ \sigma(3)}}{\alpha_{ \sigma(4)}\alpha_{\sigma(5)}\alpha_{ \sigma(6)}}$.
Thus,
the condition $\tilde\varphi\theta\tilde\varphi^{-1}=\theta \widetilde\psi$, for $\psi={\text{diag}(\beta_{ 1},\beta_{2 },\beta_{ 3},\beta_{ 4},\beta_{ 5},\beta_{ 6})}$ with $\beta_i\in\{\pm1\}$, is equivalent to  the conditions ${\alpha_i}^2= \beta_i$ for all $i$ and $\Pi_{i=1}^6\alpha_i=1$.
In particular, the automorphisms extending
 $ \varphi_1=\text{diag}(\ii,-\ii,1,1,1,1)$ and  $\varphi_2=\text{diag}(-\ii,\ii,\ii,\ii,\ii,\ii)$
 satisfy that
 $$
 \widetilde{\varphi}_1\theta{\widetilde{\varphi}_1}^{-1}=\theta \widetilde{f}_{12},\quad
 \widetilde{\varphi}_2\theta{\widetilde{\varphi}_2}^{-1}=\theta f_{-1},
 $$
 so that they belong to $\aut(\Gamma_5)$ with induced matrices in   $\W(\Gamma_5)$
$$
\begin{array}{l}
\left(
\begin{array}{c|c|c}
\bar 1& \begin{array}{cccc}\bar 1&\bar 0&\bar 0&\bar 0\end{array}  & 0 \\
\hline
0& I_4 &0 \\
\hline
0& 0 & 1
\end{array}\right),
\quad %%%% second element
\left(
\begin{array}{c|c|c}
\bar 1& 0
& \bar 1 \\
\hline
0& I_4 &0 \\
\hline
0& 0 & 1
\end{array}\right).
\end{array}
$$
The proof is ended when we multiply for previous elements in the Weyl group.
\end{proof}

%%%%%%%%%%%%%%%%%%%%%%%%%%%%%%%%%%%%%%%%%%%%%%%%%%%%%%%%%%%%%%%%%%%%%%%%%%%%%%%%%%%%%%%%%%%%%%%%%%%%%%%%%%%%%%%%%%%%%%%%

\subsection{Weyl group of the outer $\mathbb{Z}_2^3\times\mathbb{Z}^2$-grading}\label{subsec_WeylZcuadZ2cubo}

 For avoiding ambiguity, we must fix an isomorphism between the torus of $P_6$,
$\{ \tau_{\alpha,\beta}^\diamondsuit\mid \alpha,\beta\in\F^\times\} $, and $ (\F^\times)^2$.
If we notice that $\tau_{-1,-1}=\id$, and that the following is an isomorphism of two-dimensional algebraic torus
$$
\begin{array}{rcl}
(\F^\times)^2/\span{(-1,-1)}&\to&(\F^\times)^2\\
\overline{(\alpha,\beta)}&\mapsto& (\alpha\beta,\alpha/\beta),
\end{array}
$$
  hence, the automorphisms $\tau'_{\alpha\beta,\alpha/\beta }:=\tau_{\alpha,\beta}$ are more convenient to work with the grading produced by
$P_6
$, according to Equation~(\ref{eq_torosconisomorfismoprefijado}).
Denote by $\Gamma_6$ the $G_6=\Z_2\times\mathbb{Z}_2^2\times\mathbb{Z}^2$-grading induced by $P_6=\langle \Theta, {g_1}^\diamondsuit,  {g_2}^\diamondsuit \rangle \cdot \{{\tau'_{\alpha,\beta}}^\diamondsuit \mid \alpha,\beta\in\F^\times\}\le\aut(\mathbf{L})$.
Take as always the  {canonical} basis of $G_6$.

\begin{prop}
The Weyl group $\W(\Gamma_6)$ is
\begin{equation}\label{eq_ultimoweyl}
\left\{\left(
\begin{array}{c|c|c}
\bar1& a \quad b & c \quad d \\
\hline
\begin{array}{c}\bar0\\\bar0\end{array} & A & \begin{array}{c} e \quad e \\ f \quad f \end{array} \\
\hline
0 & 0 & B
\end{array}\right) \mid
  %a,b,c,d,e,f\in\mathbb{Z}_2,   A\in\GL_2(\mathbb{Z}_2), B\in  \mathcal{G} \right\}
  a,b,c,d,e,f\in\mathbb{Z}_2,   A\in\GL_2(\mathbb{Z}_2), B\in  \span{ \sigma_1,\sigma_2,-\id }\right\}
 \end{equation}
 where the set $\{\sigma_1,\sigma_2,-\id\}$
generates a group isomorphic to $\mathbb{Z}_2^2 \rtimes \mathbb{Z}_2.$
 Hence,
 $$\W(\Gamma_6)%\cong\mathbb{Z}_2^4 \rtimes \W(\Gamma_{\mathfrak{c}_4}) ESTO SEGUNDO HABRIA QUE EXPLICAR
 \cong \mathbb{Z}_2^4 \rtimes ((\mathbb{Z}_2^2 \rtimes S_3) \times (\mathbb{Z}_2^2 \rtimes \mathbb{Z}_2)).
 $$
\end{prop}

\begin{proof}
First, any element in  $\W(\Gamma_6)$ has the first column as in Equation~(\ref{eq_ultimoweyl}), since the action by conjugation maps inner   automorphisms to inner   automorphisms. Also, there is a zero block in position $(3,2)$, since the torus is invariant.

Second, if an automorphism $f\in\aut(\Gamma_6)$ has as related matrix
\begin{equation}\label{eq_laresttricccionultimocaso}\left(\begin{array}{c|c|c}
%\bar1& \begin{array}{cccc}a&b&c&d\end{array}  &e\\
\bar1& 0  &0\\
\hline
0&A& D \\
\hline
0& 0 & B
\end{array}\right)\in\W(\Gamma_6),
\end{equation}
% (necessarily $D=0$),
then, as in the proof of Proposition~\ref{prop_el5}, the restriction $f\vert_{\fix\Theta}$ belongs to the group of automorphisms of the $\mathbb{Z}_2^2\times\mathbb{Z}^2$-fine grading on $\fix\Theta\cong\mathfrak{c}_4$, and conversely.
Such Weyl group is, following  \cite[Theorem~5.7]{EK12}, isomorphic to   $(\mathbb{Z}_2^2 \rtimes \Sp_2(\Z_2)) \times (\mathbb{Z}_2^2 \rtimes \mathbb{Z}_2)$. We will provide matricial expressions in terms of our basis.
%%%%%%%%%%%%%
%%%%%%%%%%%%%
%%%%%%%%%%%%%Hago nueva versi\'{o}n a ver si me sale
Take the elements in $\Sp_8(\F)$,
$$
p_1=\left(\begin{array}{c|c}
0&I_4\\
\hline
I_4&0
\end{array}\right),
\qquad
p_2=\left(\begin{array}{c|c}
\begin{array}{cc}
0&I_2\\I_2&0
\end{array}&0\\
\hline
0&I_4
\end{array}\right),
\qquad
p_3=\left(\begin{array}{c|c}
I_4&0\\
\hline
0&\begin{array}{cc}
0&I_2\\I_2&0
\end{array}
\end{array}\right).
$$
It holds
$$
p_1\tau'_{\alpha,\beta}p_1^{-1}=\tau_{\alpha,\frac1\beta },\quad
p_2\tau'_{\alpha,\beta}p_2^{-1}=\tau_{\frac1\beta,\frac1\alpha},\quad
p_3\tau'_{\alpha,\beta}p_3^{-1}=\tau_{\beta,\alpha},
$$
so that $p_i^\diamondsuit$ is an automorphism of $\mathbf{L}$ with related matrix
$ \left(
\begin{array}{c|c|c}
\bar1& 0 & 0 \\
\hline
0 & I_2 & 0 \\
\hline
0 & 0 & *
\end{array}\right) $, and in the $(3,3)$-position
$\sigma_2$, $-\sigma_1$ and $\sigma_1$ respectively.
%%%%%%%%%%%%%
%%%%%%%%%%%%%El pedazo del A
On the other hand, $\{\ii\sigma_1,-\ii\sigma_2,\sigma_3=\sigma_1\sigma_2\}$ are the Pauli matrices and are all of them interchangeable  in $\Mat_{2\times2}(\F)$. Thus, if $p\sigma_1p^{-1}=\sigma_2$, the element $\diag(p,p,p,p)^\diamondsuit$ commutes with $\tau_{\alpha,\beta}$ and with $ \Theta$, and it applies $g_1^\diamondsuit$ into $g_2^\diamondsuit$. In this way   $\left(\begin{array}{c|c|c}
%\bar1& \begin{array}{cccc}a&b&c&d\end{array}  &e\\
\bar1& 0  &0\\
\hline
0&A& 0 \\
\hline
0& 0 &I_2
\end{array}\right)\in\W(\Gamma_6)$ for all $A\in\GL_2(\Z_2)$.
%%%%%%%%%%%%%
%%%%%%%%%%%%%EL pedazo de eeff
For the block $D$ in the $(2,3)$-position, we analyze the conjugacy class of the order two elements involved:
if we denote $g_3=g_1 g_2$, the elements of type    $2A$ are $\{g_i^\diamondsuit,(g_i\tau'_{-1,-1})^\diamondsuit,\tau'^\diamondsuit_{-1,1},\tau'^\diamondsuit_{1,-1}
\mid i=1,2,3\}$
and those of type $2B$ are $\{(g_i\tau'_{-1,1})^\diamondsuit,(g_i\tau'_{1,-1})^\diamondsuit,\tau'^\diamondsuit_{-1,-1} \mid i=1,2,3\}$. Hence the only possibility for $D$ is $\left( \begin{array}{cc} e & e \\ f & f \end{array} \right)$,
which effectively happens: note that the automorphism $\psi=
\left(\begin{array}{c|c|c}
I_4&0&0\\
\hline
0&\sigma_1&0\\
\hline
0&0&\sigma_1
\end{array}
\right)^\diamondsuit
$  commutes with   $\tau_{\alpha,\beta}$, $ \Theta$, and  $g_1^\diamondsuit$, and it verifies
$\psi g_2^\diamondsuit\psi^{-1}=(\tau_{1,-1}g_2)^\diamondsuit=(\tau'_{-1,-1}g_2)^\diamondsuit$.
We have, then, the required subgroup isomorphic to $(\mathbb{Z}_2^2 \rtimes \Sp_2(\Z_2)) \times (\mathbb{Z}_2^2 \rtimes \mathbb{Z}_2)$ and hence we must have found every element of the form (\ref{eq_laresttricccionultimocaso}).
\smallskip
%(anyway, a direct proof is not more work).

%%%%%%%%%%%%%%%%%%%%%
Third, we are left with the task of  checking that   every order two outer automorphism in $P_6 $ is conjugated to $\Theta$ by an element in the normalizer.
It is enough to check that    every order two outer automorphism in $P'_6 $ is conjugated to $\theta$ by an element in the normalizer.
Set $\varphi_0=\text{diag}(1,1,1,\ii,1,-\ii)$ and $\varphi_1=\text{diag}(\ii,-\ii,1,1,1,1)$. By the arguments in the proof of Proposition~\ref{prop_el5}, both $\widetilde{\varphi}_0$ and $\widetilde{\varphi}_1$ commute with the elements $\widetilde{g}'_1$, $\widetilde{g}'_2$ and the elements of $T_1$ and $T'_1$, and also satisfy $\widetilde{\varphi}_0\theta = \theta \widetilde{g}'_1 \widetilde{\varphi}_0$ and $\widetilde{\varphi}_1\theta = \theta \widetilde{\psi}_{-1,0} \widetilde{\varphi}_1$. Therefore, the elements of $\W(\Gamma_6)$ induced by them are, respectively,
$$
\begin{array}{l}
\left(
\begin{array}{c|c|c}
\bar 1& \bar1 \quad \bar0 & \bar0 \quad \bar0 \\
\hline
0 & I_2 & 0 \\
\hline
0 & 0 & I_2
\end{array}\right),
\quad
\left(
\begin{array}{c|c|c}
\bar1& \bar0 \quad \bar0 & \bar1 \quad \bar0 \\
\hline
 0 & I_2 & 0 \\
\hline
0 & 0 & I_2
\end{array}\right),
\end{array}
$$
and it follows that the parameters $a,b,c,d\in\mathbb{Z}_2$ can take all possible values.  
\end{proof}

 \smallskip

 Summarizing the results of the paper,

\begin{theo}\label{th_maintheorem}
The Weyl groups of the fine gradings  up to equivalence  different from the Cartan grading  on the Lie algebra of type $\e_6$ over an algebraically closed field of characteristic zero with infinite universal grading group  are:
\begin{itemize}
\item[i)] $\W(\Gamma_1)
\cong \Mat_{3\times2}(\Z_2) \rtimes (\GL_3(\mathbb{Z}_2) \times D_3)$, for $\Gamma_1$ the $\mathbb{Z}_2^3 \times\mathbb{Z}^2$-grading of type $(48, 1, 0, 7)$ described in Subsection~\ref{subsec_gradZcuadZ2cubo}.
\item[ii)] $\W(\Gamma_2)
\cong  \Z_2^4 \rtimes (\GL_3(\Z_2)\times\Z_2)$, for $\Gamma_2$ the $\mathbb{Z}_2^4 \times\mathbb{Z}$-grading
of type $(57,0,7 )$ described in Subsection~\ref{subsec_gradZZ2cuarta}.
\item[iii)] $\W(\Gamma_3)  \cong\Z_2^2\rtimes( (\Z_2^3\rtimes S_4)\rtimes S_3 )$, for $\Gamma_3$ the $\mathbb{Z}_2 \times\mathbb{Z}^4$-grading of type $(72, 1, 0, 1)$ described in Subsection~\ref{subsec_gradZcuartaZ2}.
\item[iv)] $\W(\Gamma_4)
\cong  \Z_3^2 \rtimes (\GL_2(\mathbb{Z}_3) \times D_6)$, for $\Gamma_4$ the $\mathbb{Z}_3^2\times\mathbb{Z}^2$-grading of type $(60,9 )$ described in Subsection~\ref{subsec_gradZcuadZ3cuad}.
\item[v)] $\W(\Gamma_5) \cong  (\Sp_4(\mathbb{Z}_2)\times\mathbb{Z}_2) \ltimes \mathbb{Z}_2^5$,
for $\Gamma_5$ the $\mathbb{Z}_2^5\times\mathbb{Z}$-grading of type $( 73,0,0,0,1)$ described in Subsection~\ref{subsec_gradZZ2quinta}.
\item[vi)] $\W(\Gamma_6)\cong  \mathbb{Z}_2^4 \rtimes ((\mathbb{Z}_2^2 \rtimes S_3) \times (\mathbb{Z}_2^2 \rtimes \mathbb{Z}_2))$, for $\Gamma_6$ the $\mathbb{Z}_2^3\times\mathbb{Z}^2$-grading of type $(60,7,0,1 )$ described in Subsection~\ref{subsec_gradZoutercuadZ2cubo}.
\end{itemize}
\end{theo}

%%%----------------------------------------------------------

\vspace{2.0 cm}

\end{document}